\documentclass[11pt,english]{article}
\usepackage[T1]{fontenc}
\usepackage[latin9]{inputenc}
\usepackage[letterpaper]{geometry}
\usepackage{amsmath}
\usepackage{amsthm}
\usepackage{amssymb}
\usepackage{graphicx}
\usepackage{color,hyperref}

\makeatletter
\providecommand{\tabularnewline}{\\}

\theoremstyle{plain}
\newtheorem{thm}{\protect\theoremname}
  \theoremstyle{plain}
  \newtheorem{cor}[thm]{\protect\corollaryname}
  \theoremstyle{plain}
  \newtheorem{lem}[thm]{\protect\lemmaname}

  \theoremstyle{definition}
  \newtheorem{example}[thm]{\protect\examplename}

\makeatother

\usepackage{babel}
  \providecommand{\corollaryname}{Corollary}
  \providecommand{\examplename}{Example}
  \providecommand{\lemmaname}{Lemma}
\providecommand{\theoremname}{Theorem}

\begin{document}

\title{An extrapolative approach to integration over hypersurfaces in the
level set framework}

\author{Catherine Kublik\thanks{Department of Mathematics, University of Dayton, 300 College Park, Dayton, OH 45469, USA. Email: ckublik1@udayton.edu.}~ and Richard Tsai\thanks{Department of Mathematics, KTH Royal Institute of Technology, SE-100 44, Stockholm, Sweden and Department of Mathematics and ICES, University of Texas at Austin, Austin, TX 78712, USA. Email: tsai@kth.se.}}
\maketitle
\begin{abstract}
We provide a new approach for computing integrals over hypersurfaces
in the level set framework. 
The method is based on the discretization (via simple Riemann sums) \textcolor{black}
of the \textcolor{black}{classical} formulation used in the level
set framework, with the choice of specific kernels supported on a tubular neighborhood around the interface to approximate the Dirac delta function. The \textcolor{black}{novelty lies} in the choice of kernels, specifically its number of vanishing moments,  \textcolor{black}{which enables accurate computations of integrals over \textcolor{black}{a class of closed, continuous, piecewise smooth,  curves or surfaces; e.g. curves in two dimensions that contain finite number of corners.}} We prove that for smooth
interfaces, if the kernel has enough vanishing moments (related to
the dimension of the embedding space), the analytical integral formulation coincides
exactly with the integral one wishes to calculate. For curves with
corners and cusps, the formulation is not exact but we provide an analytical
result relating the severity of the corner or cusp with
the width of the tubular neighborhood. We show numerical examples
demonstrating the capability of the approach, especially for integrating
over piecewise smooth interfaces and for computing integrals where
the integrand is only Lipschitz continuous or has an integrable singularity.
\end{abstract}

\section{Introduction}

\textcolor{black}{We propose an extrapolative approach for computing integrals over a class of piecewise smooth hypersurfaces,
given implicitly via a level set function. The method is based on the classical approximation used in the level set framework \textcolor{black}{that smears out
the Dirac $\delta$ function to a bump function with a compact support.} 
\textcolor{black}{We analyze a family of integrals over the level sets of a Lipschitz continuous function whose zero level set defines the hypersurface,  and use special kernels with vanishing moments for the approximation of the Dirac $\delta$-function. 
\emph{The novelty is in how we combine the classical formulas for the more challenging cases in which the hypersurfaces have kinks and corners.} In particular, the proposed method does not require any local explicit parameterization of the hypersurface, nor
the explicit locations of the corners and kinks on the hypersurface.
The key to our results is the smoothness of this family of integrals with respect to the parameter $\eta$ which is the signed distance between the zero level set $\Gamma_{0}$ and the "parallel" level set $\Gamma_{\eta}$. }
\textcolor{black}{We derive} an estimate of the error in terms of the severity of the corner/cusp and the width of the kernels. 
This work lays the foundation of a numerical scheme for computing general improper integrals.}

\textcolor{black}{Let $\Gamma_{0}$ be a closed hypersurface in $\mathbb R^n$ (namely, a closed curve in $\mathbb{R}^{2}$ or closed surface in $\mathbb{R}^{3}$) defined implicitly as the zero level
set of a level set function $\phi$, namely 
\begin{equation}
\Gamma_{0}:=\left\{ \mathbf{x}:\phi(\mathbf{x})=0\right\} .\label{eq:levelset}
\end{equation}
We are interested in computing integrals of the form
\begin{equation}
I_{0}:=\int_{\Gamma_{0}}f(\mathbf{x})dS(\mathbf{x}),\label{eq:surface_integral}
\end{equation}
}
\textcolor{black}{which is classically approximated in the level set framework \cite{osher_sethian88} by
\begin{equation}
S:=\int_{\mathbb{R}^{n}}\tilde{f}(\mathbf{x})\delta_{\epsilon}(\phi(\mathbf{x}))|\nabla\phi(\mathbf{x})|d\mathbf{x},
\label{eq:S}
\end{equation}
where $\tilde{f}:\mathbb R^n \rightarrow \mathbb R$ is Lipschitz, constant along the normal of $\Gamma_{0}$ and $\tilde{f} = f$ on $\Gamma_{0}$. }

\textcolor{black}{We consider $S$ as an average of a one parameter family of integrals in $\eta$, 
$$
S = \int_{\mathbb{R}^{n}}\tilde{f}(\mathbf{x})\delta_{\epsilon}(\phi(\mathbf{x}))|\nabla\phi(\mathbf{x})|d\mathbf{x} = \int_{-\epsilon}^{\epsilon}\delta_{\epsilon}(\eta) \left ( \int_{\Gamma_{\eta}}\tilde{f}(\mathbf{x})dS(\mathbf{x}) \right ) d\eta
$$
and \emph{exploit the smoothness of this family of integrals with respect to $\eta$ via suitable moment conditions on the kernel $\delta_{\epsilon}$.} Our first result shows that for smooth hypersurfaces in $\mathbb R^n$, $S$ yields the exact value of $I_{0}$ if the kernel $\delta_{\epsilon}$ has $n-1$ vanishing moments. In our second result $\Gamma_{0} \subset \mathbb R^2$ has a corner or a $p$th cusp singularity. The result states that if $\delta_{\epsilon}$ has $m$ vanishing moments and $\delta_{\epsilon}(\eta) = O(\epsilon^{-k})$, then for small $\epsilon>0$,
 $$
|S-I_{0}|=\begin{cases}
O\left (\epsilon^{2+m-k} \right ) & p=1\mbox{\mbox{ (corner)}}\\
O \left (\epsilon^{1+\frac{1}{p}-k} \right ) & p\geq2\mbox{ (cusp).}
\end{cases},
$$
\textcolor{black}{This paper relates the geometry of the hypersurface and the smoothness of its corners to the choice of kernels needed to obtain a highly accurate numerical scheme for computing integrals over that hypersurface.}
}

\textcolor{black}{\section {Motivation for the present work}
\subsection{Background and related work}
This paper provides a new understanding of surface integrals
in the level set framework, particularly in the case where the curve
or surface $\Gamma_{0}$ is only piecewise smooth. 
It is of interest to study the integration
over these types of hypersurfaces in the context of level set methods since there
are many applications of level set methods in which the hypersurfaces 
go through topological changes in  a dynamical process. In computer vision for example, the
segmentation of an image may be obtained via a two dimensional flow
using level set methods, in which case the flow will give rise to
a curve undergoing topological changes and developing corners during
its evolution. Level set methods have also been used in constrained
optimization problems \cite{maitre_santosa02,osher_santosa01}, where
the Lagrange multipliers can be expressed in terms of boundary integrals.
In addition, boundary integral methods used in combination with level sets \cite{kublik_tanushev_tsai13,kublik_tsai16,chen_kublik_tsai}
have recently shown promising results. Other applications of implicit
boundary integral methods include the computation of the Dirichlet-to-Neumannn
map in the context of shape optimization and the integration over
streamlines in fluid mechanics. In this paper, we shed light on a
mathematical framework for integrating over piecewise smooth hypersurfaces 
defined implicitly via a level set function.}

\textcolor{black}{We focus on the situation where the information about $\Gamma_{0}$
is given only via the values of a level set function $\phi$ on some grid. We shall refer
to the corresponding grid function by $\phi_{\mathbf{i}}.$ There
are two strategies for computing integrals like (\ref{eq:surface_integral})
in the level set framework:
\begin{enumerate}
\item \textcolor{black}{Derive local explicit approximations of the implicit surface, then derive quadrature rule based on the explicit, approximate surfaces. }
From the level set grid
function $\phi_{\mathbf{i}}$, one typically approximates the curve or surface $\Gamma_{0}$ by $\Gamma_{\triangle}$, (generally a polygon), and then uses this approximation to compute the integral of $f$ over $\Gamma_{0}$ using
local parameterizations of $\Gamma_{\triangle}$ (see e.g. \cite{lorensen_cline87,sacao05,smereka06}). Recently in \cite{saye15}), 
it is proposed to convert the implicit geometry into the graph of an implicitly given height function, leading to a recursive algorithm on the number of dimensions and thus requiring only one-dimensional root-finding and one-dimensional Gaussian quadratures. 
\textcolor{black}{The moment-fitting method from \cite{muller_kummer_oberlack13} requires special divergence-free bases and integration by parts. }
\item Derive an analytical integral formulation $I(f,\phi)$
that is \emph{easy} to discretize, then discretize it (see e.g. \cite{smereka06, towers07, zahedi_tornberg10, wen09, wen10, wen10_2, kublik_tanushev_tsai13, kublik_tsai16}). We note that
this approach computes the integral (\ref{eq:surface_integral}) \emph{without
using any local parameterizations of $\Gamma_{0}$. } 
\end{enumerate}
In this paper, we consider the second approach. In the level set framework,
the integral (\ref{eq:surface_integral}) is formally written in the form 
\begin{equation} 
\label{eq:classical level set integration}
\int_{\mathbb{R}^{n}}f(\mathbf{x})\delta(\phi(\mathbf{x}))|\nabla\phi|d\mathbf{x,}
\end{equation}
where $\delta(\cdot)$ denotes the Dirac delta function. \textcolor{black}{Formula \eqref{eq:classical level set integration} is then approximated
by \eqref{eq:S}, in which $\delta_\epsilon$ is a bump (kernel)  function that integrates to $1$.}  The integral \eqref{eq:S} over
$\mathbb{R}^{n}$ is then discretized via simple Riemann sums on a Cartesian grid using
a specific choice for $\delta(\cdot)$. There are many approximations
or regularizations of $\delta(\cdot)$ in the numerical literature.
Typically the regularized $\delta$-function is defined on a tubular
neighborhood around the interface of width $\epsilon$, denoted $\delta_{\epsilon}$. This effectively thickens or diffuses the interface in that tubular neighborhood. 
One choice is to take $\epsilon$ independent of the level set function
$\phi$ and the grid. In the work of Smereka \cite{smereka06}, the
discrete delta-function is concentrated within one grid cell on either
side of the interface, and is obtained by discretizing the fundamental solution of the 
Laplace equation using ghost-points. 
In the work of Towers \cite{towers07}, the discretized delta
function is computed via two different formulations involving the
Heaviside function. The more accurate formulation is obtained using
integration by parts on a suitable integral. In \cite{engquist_tornberg_tsai05},
the Dirac delta function is regularized using the gradient of the
level set function $\phi,$ a scaling that allows $\epsilon$ (the
width of the tubular neighborhood around the interface) to be small
with respect to the underlying grid}. 
\textcolor{black}{In the work of Burger etal. \cite{burger_elvetun_schlottbom15} 
the integration over a sufficiently smooth closed surface is approximated by an integral like \eqref{eq:classical level set integration} over a narrow band around the surface. The authors provide a detailed analysis of the convergence of the integral over the diffuse surface towards the integral over the surface. In particular, they provide different convergence rates depending on the smoothness of the integrand. }

\textcolor{black}{The framework for the present work started with \cite{kublik_tanushev_tsai13}
and \cite{kublik_tsai16} where the authors proposed and studied an
integral formulation over the ambient space that coincides exactly
with the line or surface integral that one wishes to calculate. This
formulation is designed for curves and surfaces that are not defined
by any explicit parameterizations and is intended to be used with
level set techniques \cite{osher_fedkiw02,sethian99} or the closest point methods, e.g. \cite{ruuth_merriman08}.
In \cite{kublik_tsai16} the formulation is provided in dimensions
two and three and extended to open curves and surfaces. }

\textcolor{black}{The integral formulation proposed in \cite{kublik_tanushev_tsai13}
and \cite{kublik_tsai16} allows the computation of integrals of the
form (\ref{eq:surface_integral}), where $\Gamma_{0}$ is the zero
level set of the signed distance function $d$ to $\Gamma$, namely
$\Gamma_{0}:=~\left\{ \mathbf{x}:d(\mathbf{x)=0}\right\} $. 
The integral formulation is given by
\begin{equation}
I_0 \equiv \int_{\mathbb{R}^{n}}f(P_{\Gamma_{0}}(\mathbf{x}))\delta_{\epsilon}(d(\mathbf{x}))J(\mathbf{x};d(\mathbf{x}))d\mathbf{x},\label{eq:intformulation_original}
\end{equation}
where $P_{\Gamma}:\mathbb{R}^{n}\mapsto\Gamma_{0}$ is the closest
point mapping to $\Gamma_{0}$ (or projection operator onto $\Gamma_{0}$)
defined as
\begin{equation}
P_{\Gamma_{0}}(\mathbf{x})=\mathbf{x}-d(\mathbf{x})\nabla d(\mathbf{x}),\label{eq:proj_change_var}
\end{equation}
$\delta_{\epsilon}$ is an averaging kernel 
specifying a tubular neighborhood around $\Gamma_{0}$,
and $J(\mathbf{x};d(\mathbf{x}))$ is the Jacobian that accounts for
the change in curvature between nearby level sets and the zero level
set $\Gamma_{0}$. The main advantage of formulation (\ref{eq:intformulation_original})
is that unlike \eqref{eq:classical level set integration} which is an \emph{approximation} of $I_{0}$ using a regularized Dirac-$\delta$
function concentrated on $\Gamma_{0}$ \cite{dolbow_harari09,engquist_tornberg_tsai05,smereka06,towers07,zahedi_tornberg10},
(\ref{eq:intformulation_original}) is equal to $I_{0}$ analytically.
Errors are therefore due only to the numerical scheme used to discretize
(\ref{eq:intformulation_original}) instead of both the numerical
scheme and the anterior approximation.}

\textcolor{black}{For smooth curves or surfaces, the integral formulation (\ref{eq:intformulation_original})
is very powerful: it provides a very elegant, simple and attractive
computational method for computing surface integrals. In addition,
the authors in \cite{kublik_tsai16} showed that the Jacobian can
be expressed as the product of the non zero singular values of the
Jacobian matrix of the closest point mapping $P_{\Gamma_{0}}$. The
benefit of such an expression is that for smooth integrands and smooth
curves or surfaces, the accuracy of the discretizations of (\ref{eq:intformulation_original})
will depend only on the order of the finite difference scheme used
to approximate the Jacobian matrix of $P_{\Gamma_{0}}$. 
}

\textcolor{black}{Finally, we remark
that if one has a level set function $\phi$ which is not the signed
distance function to $\Gamma_{0}$, fast algorithms such as fast marching
and fast sweeping \cite{cheng_tsai08,russo_smereka00,sethian96,tsai_cheng_osher_zhao03,tsitsiklis95}
can be used to construct a signed distance function $d$ to $\Gamma_{0}$.
}
\subsection{Computational difficulties near a corner}

Both Type I and Type II methods have difficulties resolving singularities
from only the values of $\phi_{\mathbf{i}}.$ In particular, Type
II methods using the regularization parameter $\epsilon$ lead to
$O(\epsilon)$ error for each corner. This error is partially due
to the discretization of the Jacobian term. The new approach discussed
in this paper does not use the Jacobian term and as such gives a viable
approach to handling integrable singularities.

The particular approach described in \cite{kublik_tanushev_tsai13,kublik_tsai16}
has specific difficulties in resolving the singularities. Indeed,
in a neighborhood of a singular point, the change of variables (\ref{eq:proj_change_var})
breaks down. This occurs whenever the signed distance function is
not $C^{1,\alpha},\alpha>1$, and when the reach of the distance function
is smaller than the tubular neighborhood $\left\{ \mathbf{x}:-\epsilon\leq d(\mathbf{x})\leq\epsilon\right\} $(the
reach refers to the region near $\Gamma_{0}$ where $d$ is differentiable).
In addition, since the expression for the Jacobian $J$ involves curvatures
of level sets, it will be necessary to use one-sided differencing
to discretize $J$ in order to avoid differentiating $d$ across kinks.
Finally, if we consider a corner in two dimensions (see Figure \ref{fig:corner-and-projection})
we see that in the convex region enclosed by $\Gamma_{0}$ near the
corner point, points on an $\eta$-level set will not project onto
the entire curve $\Gamma_{0}$, whereas points on the $(-\eta$)-level set will all project onto $\Gamma_{0}$, leading to a ``deficiency''
of points coming from one side. 

\paragraph*{A characterization of corners and edges via the closest point projection}

Let us now use the closest point mapping with distance $\eta$ more
explicitly as $P_{\eta}(\mathbf{x}):=\mathbf{x}-\eta\nabla d(\mathbf{x})$.
One way to circumvent this ``deficiency'' issue is to identify the
points on the $(-\eta$)-level set that project onto the part of $\Gamma_{0}$
that is missed by projecting the points from the $\eta$ level set.
Such points satisfy

\begin{equation}
P_{-\eta}(P_{2\eta}(\mathbf{x}))\neq P_{\eta}(\mathbf{x}),\eta=d(\mathbf{x}),\label{eq:overproject}
\end{equation}
which states that if we over project $\mathbf{x}$ by a distance $2\eta$
and then project the result back by a distance $\eta$ (in the opposite
direction), we are not back at the same point. Points away from the
corner on the other hand, i.e. in smooth regions, satisfy $P_{-\eta}(P_{2\eta}(\mathbf{x}))=P_{\eta}(\mathbf{x})$.
See Figure \ref{fig:corner-and-projection}. 

\begin{figure}[h]
\begin{centering}
\includegraphics[scale = 1]{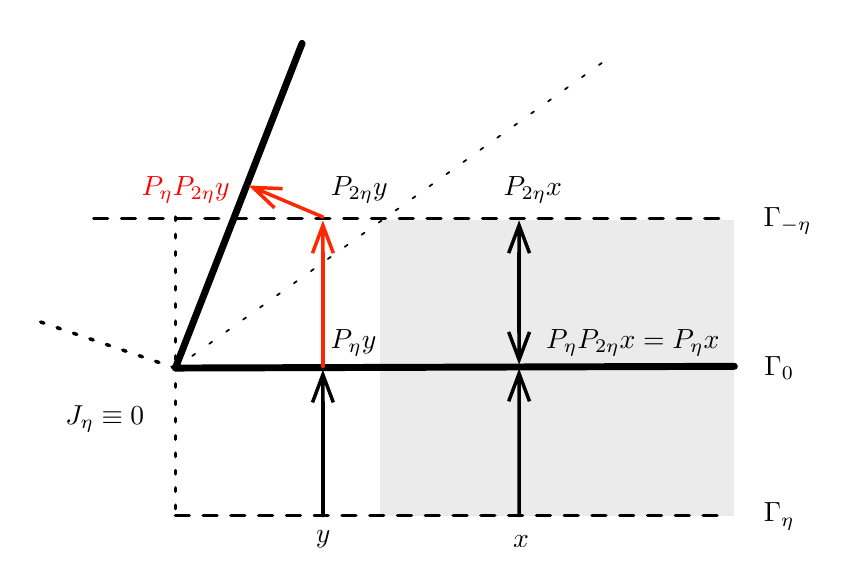}
\par\end{centering}
\caption{Projection near a corner. \label{fig:corner-and-projection}}
\end{figure}

Thus, once we have identified
the points that satisfy (\ref{eq:overproject}) , we count these points
twice to compensate for the fact that these points have no corresponding
ones on the $\eta$ level set. This effectively translates into the
following correction integral:
\[
\int_{\mathbb{R}^{2}}f(P_{\Gamma}(\mathbf{x}))\omega(\mathbf{x})\delta_{\epsilon}(d(\mathbf{x}))d\mathbf{x},
\]
with the weight $\omega$ defined as
\[
\omega(\mathbf{x})=\begin{cases}
0 & |d(\mathbf{x})\Delta d(\mathbf{x})|=1\\
2 & P_{-\eta}(P_{2\eta}(\mathbf{x}))=P_{\eta}(\mathbf{x}),\eta=d(\mathbf{x})\\
1 & \mbox{otherwise.}
\end{cases}
\]
There are several issues with this approach: the first one is that
the modification of the integrand with the weight $\omega$ leads
to a discontinuous integrand. Thus the numerical approximation of
the integral will not be able to reach a high order of accuracy. Second,
there are numerical difficulties in the implementation of criteria
(\ref{eq:overproject}): numerically, this requires the use of a threshold,
which in turns raises the question of how to choose such a threshold
value. Thus, this approach is not suitable for high accuracy and does
not provide a seamless implementation.

The purpose of the present work is to give an alternative but related
integral formulation that allows the computations of surface integrals
where the curve or surface has singular points, and which does not
suffer from the issues discussed above. A large advantage of the new
approach compared to (\ref{eq:intformulation_original}) is that the
new formulation does not involve the Jacobian $J$ and is therefore
more convenient to use. Additionally, it provides a mathematical understanding
of the relationship between accuracy and how severe a singularity
is. 

\section{The extrapolative approach}

We now present the new approach which explores the smoothly varying
relations among the different level sets of $\phi$ near $\Gamma_{0}.$ 
In particular, for surfaces having corners, the integration of the nearby  parallel surfaces varies smoothly 
as a function of the distance to the surface $\Gamma_0$, except at distance $0$ 
(corresponding to the integral on $\Gamma_0$).
Hence, it is possible to use kernels having suitable properties to approximate integration on $\Gamma_0$ by
\emph{extrapolating} integrations defined on other nearby surfaces.

\subsection{Smooth curves and surfaces}

Let $\phi:\mathbb{R}^{n}\mapsto\mathbb{R}$, $n \in \mathbb N$, be a Lipschitz
function and $\Gamma_{\eta}:=\left\{ \mathbf{x:\phi(\mathbf{x})=\eta}\right\} $
be the $\eta-$level set of $\phi$. We consider $\tilde{f}:\mathbb{R}^{n}\mapsto\mathbb{R}$
to be a Lipschitz function and define $S$ as

\begin{equation}
S:=\int_{\mathbb{R}^{n}}\tilde{f}(\mathbf{x})\delta_{\epsilon}(\phi(\mathbf{x}))|\nabla\phi(\mathbf{x})|d\mathbf{x}.\label{eq:integral_noJacobian}
\end{equation}

Integrals of the form (\ref{eq:integral_noJacobian}) have been used
to approximate $I_{0}$ but unlike (\ref{eq:intformulation_original})
which coincides exactly with $I_{0}$, we have in general $S\approx I_{0}.$
However, under specific conditions which we explain below, it is possible
to have $S=I_{0}$ or to know precisely how the error between $S$
and $I_{0}$ behaves in terms of $\epsilon$ (width of the tubular
neighborhood around $\Gamma_{0})$ for example and in terms of a corner
or how sharp a cusp is. 

We define the one-parameter family of functionals

\begin{equation}
I[\tilde{f},\phi](\eta):=\int_{\Gamma_{\eta}}\tilde{f}(\mathbf{x})dS(\mathbf{x}),\label{eq:familyfunc}
\end{equation}
which represents the integral of $\tilde{f}$ over the $\eta-$level
set of $\phi.$ It is worth pointing out that in \cite{kublik_tanushev_tsai13,kublik_tsai16},
the authors considered a similar approach to construct (\ref{eq:intformulation_original}),
but their family of functionals was $F_{\eta}:=\int_{\Gamma_{\eta}}\tilde{f}(\mathbf{x})J_{\eta}(\mathbf{x})dS(\mathbf{x})$,
where the Jacobian $J_{\eta}$ is the same as the Jacobian $J$ in
(\ref{eq:intformulation_original}). The purpose of this Jacobian
was to ensure the equality $F_{\eta}=I_{0}$ for all $-\epsilon\leq\eta\leq\epsilon$.
In other words, $I_{0}$ was parameterized in terms of the nearby level
sets within the tubular neighborhood. Unlike this original approach,
(\ref{eq:familyfunc}) is not equal to $I_{0}$ for any $\eta$ since
there is no Jacobian term to compensate for the change in curvature.
Now by the coarea formula, we can average over the parameterizations
(\ref{eq:familyfunc}) using an averaging kernel $\delta_{\epsilon}$
to obtain
\begin{align}
\int_{-\epsilon}^{\epsilon}\delta_{\epsilon}(\eta)I[\tilde{f},\phi](\eta)d\eta & =\int_{\mathbb{R}^{n}}\tilde{f}(\mathbf{x})\delta_{\epsilon}(\phi(\mathbf{x}))|\nabla\phi(\mathbf{x})|d\mathbf{x}=S.\label{eq:coarea}
\end{align}

We then have the following result:
\begin{thm}
\label{thm:Theorem1} Suppose that

\begin{enumerate}
\item $d$ is the signed distance function to $\Gamma_{0}$, i.e. $|\nabla d|=1$.
\item $\nabla\tilde{f}\cdot\nabla d=0$ in the viscosity solution sense,
meaning that $\tilde{f}$ is constant along the normals of $\Gamma_{0}$ wherever
normals are well defined (namely $\tilde{f}$ is the constant extension
of $f:\Gamma_{0}\mapsto\mathbb{R}$ along the normals of $\Gamma_{0}$).
\item $\Gamma_{\eta}$ are closed $C^{2}$ hypersurfaces for $-\epsilon\leq\eta\leq\epsilon.$
\end{enumerate}
Then for sufficiently small $\epsilon>0$ such that $\Gamma_{\pm\epsilon}\neq\emptyset$,
we have
\[
I[\tilde{f},d](\eta)=I_{0}+\sum_{i=1}^{n-1}A_{i}\eta^{i},
\]
where $A_{i}$, $1\leq i\leq n$ are constants that depend on $\tilde{f}$
and $d$.

\end{thm}
\begin{proof}
Let's denote the closest point mapping to $\Gamma_{\eta}$ (aka the
projection operator onto $\Gamma_{\eta}$ ) as $P_{\Gamma_{\eta}}:\mathbb{R}^{n}\mapsto\Gamma_{\eta}$.
If $\Gamma_{\eta}\in C^{2}(\mathbb{R}^{n})$ for all $-\epsilon\leq\eta\leq\epsilon$,
$\epsilon>0$, then
\begin{align*}
I[\tilde{f},d](\eta) & =\int_{\Gamma_{\eta}}\tilde{f}(\mathbf{x})dS(\mathbf{x})\\
 & =\int_{\Gamma_{0}}\tilde{f}(P_{\Gamma_{\eta}}(\mathbf{x}))\mathcal{J_{\eta}}(\mathbf{x})dS(\mathbf{x}),
\end{align*}
where $\mathcal{J_{\eta}}$ is the Jacobian that accounts for the
change in curvature between $\Gamma_{\eta}$ and $\Gamma_{0}$. Using
the results from \cite{kublik_tanushev_tsai13}, $\mathcal{J_{\eta}}$
can be written as $\mathcal{J_{\eta}}=1+\sum_{i=1}^{n-1}\sigma_{i}(h)\eta^{i}$,
where $\sigma_{i}(h)$ is the symmetric polynomial in the eigenvalues
of the Weingarten map induced by the second fundamental form $h$
associated to $\Gamma_{\eta}$. Thus we have
\begin{align*}
I[\tilde{f},d](\eta) & =\int_{\Gamma_{0}}\tilde{f}(P_{\Gamma_{\eta}}(\mathbf{x}))\left(1+\sum_{i=1}^{n-1}\sigma_{i}(h)\eta^{i}\right)dS(\mathbf{x})\\
 & =\int_{\Gamma_{0}}\tilde{f}(P_{\Gamma_{\eta}}(\mathbf{x}))dS(\mathbf{x})+\sum_{i=1}^{n-1}\eta^{i}\int_{\Gamma_{0}}\tilde{f}(P_{\Gamma_{\eta}}(\mathbf{x}))\sigma_{i}(h)dS(\mathbf{x})\\
 & =\int_{\Gamma_{0}}f(x)dS(\mathbf{x})+\sum_{i=1}^{n-1}\left(\int_{\Gamma_{0}}\tilde{f}(P_{\Gamma_{\eta}}(\mathbf{x}))\sigma_{i}(h)dS(\mathbf{x})\right)\eta^{i}\\
 & =I_{0}+\sum_{i=1}^{n-1}A_{i}\eta^{i},
\end{align*}
where $A_{i}=\int_{\Gamma_{0}}\tilde{f}(P_{\Gamma_{\eta}}(\mathbf{x}))\sigma_{i}(h)dS(\mathbf{x}))$,
$1\leq i\leq n$, and $I_{0}=\int_{\Gamma_{0}}\tilde{f}(P_{\Gamma_{\eta}}(\mathbf{x}))dS(\mathbf{x})$
since $\tilde{f}$ is constant along the normals to $\Gamma_{0}$,
and $P_{\Gamma_{\eta}}$ is the orthogonal projection onto $\Gamma_{\eta}$, leading to $\tilde{f}(P_{\Gamma_{\eta}}(\mathbf{x}))=f(\mathbf{x})$,
for all $\mathbf{x}\in\Gamma_{0}.$ 
\end{proof}
\textcolor{black}{
This leads to the following generalization. 
\begin{cor} 
\label{cor:Corollary0}
Assume now that the level set function $\phi$ is given by $\phi(\mathbf{x}) = \psi(d(\mathbf{x}))$, where $d$ is the signed distance function to $\Gamma_{0}$ and  $\psi : \mathbb R \rightarrow \mathbb R$ is a strictly monotonic function in $[-\epsilon,\epsilon]$, $\epsilon>0$ satisfying $\psi(0) = 0$.
Then 
$$
I[\tilde{f},\phi](\eta) = I_{0} + \sum_{i=1}^{n-1} B_{i} \left (\psi^{-1}(\eta) \right )^i,
$$
where $B_i, 1 \leq i \leq n$ are constants that depend on $\tilde{f}$ and $\phi$. 
\end{cor}
\begin{proof}
Let $-\epsilon \leq \eta \leq \epsilon$ and let $\Gamma_{\eta} =\left \{ \mathbf{x} : \phi(\mathbf{x}) = \eta \right \}$. Since $\psi$ is strictly monotone in $[-\epsilon,
\epsilon]$, $\psi$ is invertible such that 
$\Gamma_{\eta} = \left \{ \mathbf{x} : \psi(d(\mathbf{x})) = \eta \right\} = \left \{ \mathbf{x} : d(\mathbf{x}) = \psi^{-1}(\eta) \right\} = \left \{ \mathbf{x} : d(\mathbf{x}) = \xi \right\}$,
where $\xi = \psi^{-1}(\eta)$ and $d$ is the signed distance function to $\Gamma_{0}$. Thus
\begin{align*}
I[\tilde{f},\phi](\eta) & =\int_{\Gamma_{\eta}}\tilde{f}(\mathbf{x})dS(\mathbf{x})\\
& = \int_{\Gamma_{\xi}} \tilde{f}(\mathbf{x})dS(\mathbf{x}) \\
& = I[\tilde{f},d](\xi) \\
& = I_{0} + \sum_{i=1}^{n-1} B_{i} \xi^i,
\end{align*}
by using Theorem~\ref{thm:Theorem1}. Therefore
$$
I[\tilde{f},\phi](\eta) = I_{0} + \sum_{i=1}^{n-1} B_{i} \left (\psi^{-1}(\eta) \right )^i,
$$
where $B_{i} = \int_{\Gamma_{0}} \tilde{f}(P_{\Gamma_{\xi}}) \sigma_{i}(h) dS(\mathbf{x})$, $1 \leq i \leq n$ with $\xi = \psi^{-1}(\eta)$. 
\end{proof}
An example of such a level set function is $\phi(\mathbf{x}) = d(\mathbf{x})^3$ or $\phi(\mathbf{x}) = \text{sgn}(d) d^2(\mathbf{x})$ where $d$ is the signed distance function to $\Gamma_{0}$. We note that in these two examples, the expansion in $\eta$ in $I[\tilde{f},\phi](\eta)$ will include fractional powers of $\eta$; Indeed if $\phi(\mathbf{x}) = d(\mathbf{x})^3$, then $\xi = \eta^{\frac{1}{3}}$ and thus
$$
I[\tilde{f},\phi](\eta) = I_{0} + \sum_{i=1}^{n-1} B_{i} \eta^\frac{i}{3}.
$$
}

In the dimensions of interest ($n=2,3$), Theorem~\ref{thm:Theorem1} states
that if $\Gamma_{0}$ is $C^{2}$, $\tilde{f}$ is constant along
the normals of $\Gamma_{0}$ and $d$ is the signed distance function
to $\Gamma_{0}$, then $I[\tilde{f},d]$ is linear in $\eta$ for
curves in two dimensions and quadratic in $\eta$ for surfaces in
three dimensions. Therefore, if we average the parameterizations $I[\tilde{f},d]$
with a kernel $\delta_{\epsilon}$ that has enough vanishing moments,
the terms in $\eta$ will vanish and we will be left with $I_{0}$,
thus making $S=I_{0}.$ The result is stated in the following Theorem:
\begin{thm}
\label{thm:Theorem2} Assume that Theorem \ref{thm:Theorem1} holds
and assume that the averaging kernel $\delta_{\epsilon}$ is compactly
supported in $[-\epsilon,\epsilon]$ with $n-1$ vanishing moments
(where $n$ is the dimension), namely
\[
\int_{-\infty}^{\infty}\delta_{\epsilon}(\eta) \eta^p d\eta=\begin{cases}
1 & p=0,\\
0 & 0<p\leq n-1,
\end{cases}
\]
then 
\[
I_{0}=\int_{\Gamma_{0}}f(x)dS(x)=\int_{\mathbb{R}^{n}}\tilde{f}(\mathbf{x})\delta_{\epsilon}(d(\mathbf{x}))d\mathbf{x}=S.
\]
\end{thm}

\begin{proof}
Using (\ref{eq:coarea}), the result of Theorem (\ref{eq:intformulation_original})
and the assumptions on $\delta_{\epsilon},$ we have
$$
S=\int_{-\epsilon}^{\epsilon}\delta_{\epsilon}(\eta)I[\tilde{f},d](\eta)d\eta=\int_{-\epsilon}^{\epsilon}\delta_{\epsilon}(\eta) \left (I_{0}+\sum_{i=1}^{n-1}A_{i}\eta^{i} \right )d\eta 
$$
$$
=I_{0}\int_{-\infty}^{\infty}\delta_{\epsilon}(\eta)d\eta +\sum_{i=1}^{n-1}A_{i}\int_{-\infty}^{\infty}\delta_{\epsilon}(\eta)\eta^{i}d\eta=I_{0}.
$$

\end{proof}
The main upshot of this result is that if the curve or surface is
smooth (i.e. $C^{2}$), it is possible to construct $S$ such that
it coincides exactly with $I_{0}$. To be more specific, if the kernel
$\delta_{\epsilon}$ has enough vanishing moments, then $S=I_{0}.$
For the dimensions of interest ($2$ and $3)$, it is easy to construct
kernels with $1$ or $2$ vanishing moments. For large dimensions,
we point out that it might not be easy to construct a kernel with
enough vanishing moments to obtain the equality between $S$ and $I_{0}$,
but in that case, the error between $S$ and $I_{0}$ will be related
to the number of vanishing moments of $\delta_{\epsilon}$. Thus,
the higher the number of vanishing moments, the more accurate the
approximation $S$ will be to $I_{0}$. 

Note that in general, $I[\tilde{f},d]$ may not be a polynomial
in $\eta$, but as long as $I[\tilde{f},d]$ has a Taylor expansion
about $\eta=0$, the accuracy of using $S$ to approximate $I_{0}$
will be determined by the number of vanishing moments of the kernel
$\delta_{\epsilon}$.  \textcolor{black}{Also in the case of a general level set function $\phi$, $I[\tilde{f},\phi]$ will have an expansion with fractional powers of $\eta$, thus requiring a special class of kernels with "fractional" vanishing moments in order to expect $S$ to coincides exactly with $I_{0}$. This, highlights that for smooth interfaces, it is advantageous to use a signed distance function.}

Theorem~\ref{thm:Theorem2} has several implications. First, it is very convenient
to use because unlike (\ref{eq:intformulation_original}), this formulation
does not need the Jacobian term. It is therefore simpler to implement
and performs the same as (\ref{eq:intformulation_original}) \emph{as
long as the kernel is chosen appropriately}. In a way, this can be
understood as a trade-off between number of vanishing moments and
Jacobian. The Jacobian allows (\ref{eq:intformulation_original})
to be exact, but $S$ can be made exact by using a kernel with enough
vanishing moments. Second, this simpler formulation gives a viable
approach for approximating integrals over curves and surfaces with
singularities. 

\subsection{Curves with corners}

In this section, we assume that $\Gamma_{0}$ is a \textcolor{black}{piecewise $C^2$} closed curve in
$\mathbb{R}^{2}$ with a corner at $(x_{0},y_{0})$ The purpose of
this section is to study the behavior of $I[\tilde{f},\phi](\eta)$
as a function of $\eta$ around $\eta=0$ in order to deduce how the
error incurred between $S$ and $I_{0}$ depends on the type of singularity
(corner or cusp). We describe the corner/cusp as follows:

\begin{itemize}
\item Corner. We consider a function $g:[0,\infty)\mapsto[0,\infty)$ such that $g(0)=0$ and
for $x>0,$ $g$ is $C^{2}$ with $g'(0)>0$.  

\item Cusp. We consider a function $g:[0,\infty)\mapsto[0,\infty)$ such that $g^{(\nu)}(0)=0$ for $0\le\nu<p$ ($p \in \mathbb N, p \geq 2$) \textcolor{black}{and $g^{(p)}(0) > 0$.}
\end{itemize}

\textcolor{black}{We consider the
part of $\Gamma_{0}$ around the corner to be defined by the union
of the graphs of $g$ and $-g$. In this case, both sides of the $x$-axis contribute the same amount of ``singularity". In the general case, where the corner is likely not symmetric about the $x$-axis, the error will be dominated by the most ``singular side".}'

In this new coordinate system, the
corner is at $(0,0)$. We then parameterize the part of $\Gamma_{0}$
that lies above the $x-$ axis by 
\[
\gamma(x):=\left(\begin{array}{c}
x\\
g(x)
\end{array}\right),
\]
as shown in Figure~\ref{fig:corner-diagram}. In this setup, the distance
function to $\Gamma_{0}$ is differentiable away from the positive
$x$-axis. 

Define the normal to $\gamma$ by 
\[
\vec{n}(x)=\frac{1}{\sqrt{1+g'(x)^{2}}}\left(\begin{array}{c}
g'(x)\\
-1
\end{array}\right)
\]
and consider the lines
\[
L(\eta;x):=\gamma(x)+\eta\vec{n}(x),
\]
which correspond to the characteristics of the Eikonal equation that
emanates from $\gamma$ before they intersect the $x-$axis. We first
find $\eta$ such that $L(\eta,x)$ intersects the $x-$axis, i.e.
\[
\left(\begin{array}{c}
x\\
g(x)
\end{array}\right)+\frac{\eta}{\sqrt{1+g'(x)^{2}}}\left(\begin{array}{c}
g'(x)\\
-1
\end{array}\right)=\left(\begin{array}{c}
X_{\eta}\\
0
\end{array}\right),
\]
where $X_{\eta}$ is the $x-$coordinate of the intersection point
between the $\eta-$level set and the $x-$axis. (See Figure~\ref{fig:corner-diagram}).

\begin{figure}[h]

\begin{centering}
\includegraphics[scale = 1]{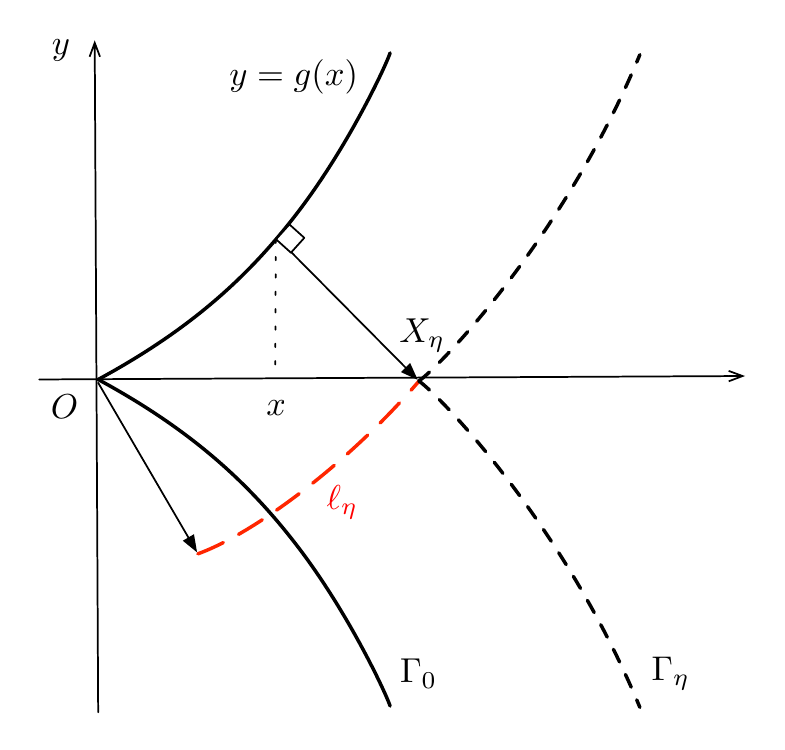}
\par\end{centering}
\caption{An illustration of a curve with a corner/cusp and a curve which is
$\eta$ distance away from it.\label{fig:corner-diagram}}

\end{figure}

Using the information
from the $y-$coordinates in the above vector equation we obtain
\begin{equation}
\eta=g(x)\sqrt{1+g'(x)^{2}}:=F(x).\label{eq:definition_eta_F}
\end{equation}
Thus, the $\eta(x)$-level set of the signed distance function to
$\Gamma_{0}$ has a corner or cusp at $(X_{\eta},0)$. Now we estimate
the length of the portion of $\Gamma_{0}$, the projection of which
to the $\eta(x)-$level set is missing, denoted by $l(x)$ and given
by 
\begin{equation}
l(x)=\int_{0}^{x}\sqrt{1+g'(s)}ds.
\end{equation}

If we now look at the integration of $\tilde{f}$ over a portion of
the $\eta-$level set of $d$ above the $x-$axis, we are missing
the corresponding integral
\textcolor{black}{
$$
l_{\eta(x)}^{+}[\tilde{f},d]  =\int_{0}^{x(\eta)}\tilde{f}(s,g(s))\sqrt{1+g'(s)^{2}}J_{s}ds,
$$
where $J_{s} = (1+\eta\kappa(s))$ is the Jacobian that accounts for the change in curvature between $\Gamma_{0}$ and $\Gamma_{\eta}$ (see \cite{kublik_tanushev_tsai13}). In this integral, we are integrating over the portion of $\Gamma_{\eta}$ that would be there if there was no corner (represented by the long dashed curve in Figure~\ref{fig:corner-diagram}). The term $\sqrt{1+g'(s)^2} ds$ is the arclength along $\Gamma_{0}$, and the term $\sqrt{1+g'(s)^2} J_{s}ds$ is the arclength along $\Gamma_{\eta}$. It follows that we have}

\begin{align*}
l_{\eta(x)}^{+}[\tilde{f},d]  & =\int_{0}^{x(\eta)}\tilde{f}(s,g(s))(1+\eta\kappa(s))\sqrt{1+g'(s)^{2}}ds\\
 & =\int_{0}^{x(\eta)}\tilde{f}(s,g(s))\sqrt{1+g'(s)^{2}}ds+\eta\int_{0}^{x(\eta)}\tilde{f}(s,g(s))\kappa(s)\sqrt{1+g'(s)^{2}}ds\\
 & =A^{+}(x(\eta))+\eta B^{+}(x(\eta)),
\end{align*}
where $\kappa(s)$ is the curvature of $\Gamma_{0}$ at the point
$(s,g(s)),$ and $x(\eta)=F^{-1}(\eta)$ where $F$ is given in
(\ref{eq:definition_eta_F}) and $F$ is invertible around $x=0$.
The invertibility of $F$ is proven in Lemma \ref{lem:xofeta}. Away
from the corner, the curve is smooth and therefore the error between
$S$ and $\mbox{\ensuremath{I_{0}}}$ is dominated by the effect of
the corner. We choose to focus on the corner for $x\in[0,b]$, $b>0.$ (See Figure~\ref{fig:corner-diagram-3}). 

\begin{figure}[h]

\begin{centering}
\includegraphics[scale = 1]{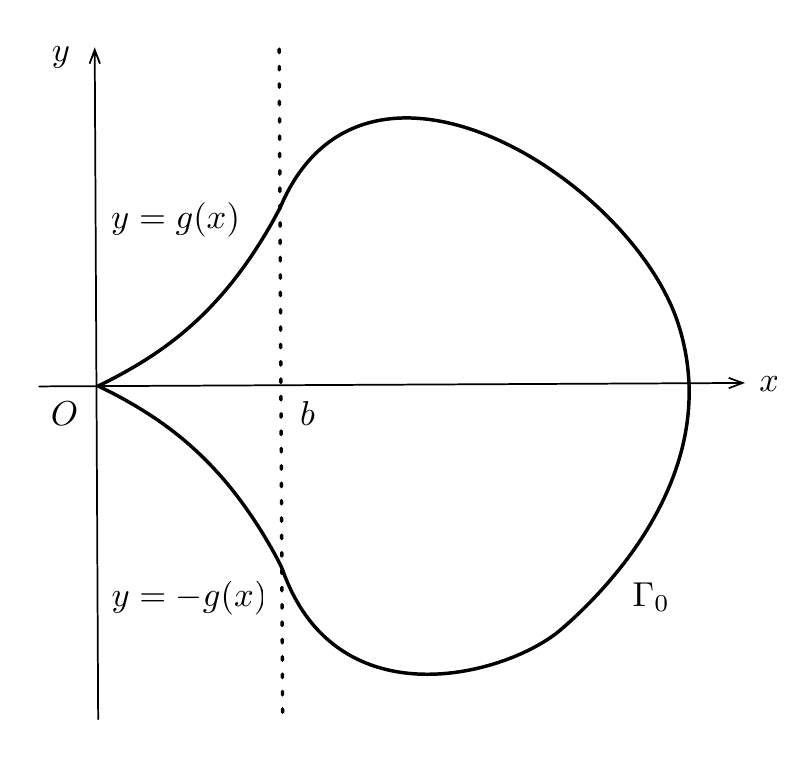}
\par\end{centering}
\caption{The closed curve $\Gamma_0$ with a corner at $(0,0)$.\label{fig:corner-diagram-3}}

\end{figure}

Thus the integral over one side of the
corner (above the $x-$axis) is 
\begin{align*}
G_{\eta}^{+} & =\int_{0}^{b}\tilde{f}(s,g(s))(1+\eta\kappa(s))\sqrt{1+g'(s)^{2}}ds\\
 & =\int_{0}^{x(\eta)}\tilde{f}(s,g(s))(1+\eta\kappa(s))\sqrt{1+g'(s)^{2}}ds+\int_{x(\eta)}^{b}\tilde{f}(s,g(s))(1+\eta\kappa(s))\sqrt{1+g'(s)^{2}}ds\\
 & =l_{\eta}^{+}+I_{\eta}^{+},
\end{align*}
where $I_{\eta}^{+}$ is the integration of $\tilde{f}$ along the
portion of $\Gamma_{\eta}$ above the $x-$axis. Thus
\[
I_{\eta}^{+}=G_{\eta}^{+}-l_{\eta}^{+}=G_{\eta}^{+}-A^{+}(x(\eta))-\eta B^{+}(x(\eta)).
\]

The result for the integral along the portion of $\Gamma_{\eta}$
below the $x-$axis is obtained similarly and thus
\[
I_{\eta}^{-}=G_{\eta}^{-}-l_{\eta}^{-}=G_{\eta}^{-}-A^{-}(x(\eta))-\eta B^{-}(x(\eta)),
\]
with 
\[
A^{-}(x(\eta))=\int_{0}^{x(\eta)}\tilde{f}(s,-g(s))\sqrt{1+g'(s)^{2}}ds,
\]
\[
B^{-}(x(\eta))=\int_{0}^{x(\eta)}\tilde{f}(s,-g(s))\kappa(s)\sqrt{1+g'(s)^{2}}ds.
\]

We point out that $l_{\eta}^{+}$ and $l_{\eta}^{-}$ are no longer
polynomials in $\eta.$ WLOG, we assume that $\Gamma_{0}$ has only
one corner and we denote by $\mathcal{P}_{\eta}^{+}$ and $\mathcal{P}_{\eta}^{-}$
the portion of $\Gamma_{\eta}$ above the $x-$axis and below the
$x-$axis respectively. By construction, the term $G_{\eta}^{+}+G_{\eta}^{-}+\int_{\Gamma_{0}\setminus(\mathcal{P}_{\eta}^{+}\cup\mathcal{P}_{\eta}^{-})}\tilde{f}(\mathbf{x})dS(\mathbf{x})$
does not have any corner any more. Thus we can apply Theorem \ref{thm:Theorem1}
and obtain 
\begin{equation}
G_{\eta}^{+}+G_{\eta}^{-}+\int_{\Gamma_{0}\setminus(\mathcal{P}_{\eta}^{+}\cup\mathcal{P}_{\eta}^{-})}\tilde{f}(\mathbf{x})dS(\mathbf{x})=I_{0}+A\eta,\mbox{ }A\in\mathbb{R},\label{eq:constA}
\end{equation}
which is equivalent to
\begin{equation}
I[\tilde{f},d](\eta):=I_{\eta}^{+}+I_{\eta}^{-}+\int_{\Gamma_{0}\setminus(\mathcal{P}_{\eta}^{+}\cup\mathcal{P}_{\eta}^{-})}\tilde{f}(\mathbf{x})dS(\mathbf{x})=I_{0}+A\eta-l_{\eta}^{+}-l_{\eta}^{-}.\label{eq:eq_corner}
\end{equation}
Note that the expression for the constant $A$ in (\ref{eq:constA})
can be obtained by using the expressions for the constants $A_{i}$
given in the proof of Theorem \ref{thm:Theorem1}. 

Not surprisingly, it turns out that there is a fundamental difference in integration of parallel surfaces
near a corner and near a cusp on $\Gamma_0$. 
\begin{thm}
\label{thm:MainResult}Consider a curve $\Gamma_{0}$ in $\mathbb{R}^{2}$
such that $\Gamma_{0}$ has a corner at $(x_{0},y_{0})$.  \textcolor{black}{Let $d$ be the signed distance function to $\Gamma_{0}$ used to compute $S$}. Assume $f\in C^{0}(\Gamma_{0})$
and that the curvature $\kappa$ is continuous everywhere except at
the corner point. Assume that $g\in C^{2}([0,\infty),[0,\infty))$ with $g(0)=0$ and for $p\in\mathbb{N}$,
$g^{(\nu)}(0)=0$ for $0\le\nu<p$ and \textcolor{black}{$g^{(p)}(0)>0$}, such that the corner/cusp is modeled by the
graphs of $g$ and $-g$. In this new coordinate system, the corner/cusp
is at $(0,0).$ Suppose also that the averaging kernel $\delta_{\epsilon}$
is compactly supported in $[-\epsilon,\epsilon]$ with $m$ vanishing
moments such that 
 \[
\delta_{\epsilon}(\eta)=O(\epsilon^{-k})
\]
 as $\eta\rightarrow0$, ($k\in\mathbb{N}$).  Then 
$$
|S-I_{0}|=\begin{cases}
O\left (\epsilon^{2+m-k} \right ) & p=1\mbox{\mbox{ (corner)}}\\
O \left (\epsilon^{1+\frac{1}{p}-k} \right ) & p\geq2\mbox{ (cusp)}
\end{cases},
$$
 for small $\epsilon>0$. 
\end{thm}
Note that the number of vanishing moments of the kernel $\delta_{\epsilon}$
only plays a role in the case of a corner. For cusps, it is necessary
to construct a different class of kernels that integrate to zero when
multiplied by fractional powers of $\eta$. 

\section{Proof of Theorem \ref{thm:MainResult}}

We start with two technical lemmas that are needed to complete the
proof of Theorem \ref{thm:MainResult} and then give the proof of
Theorem \ref{thm:MainResult}.
\begin{lem}
\label{lem:UnifConv}Suppose $p\in\mathbb{N}$, and assume that $k:[0,\infty) \mapsto\mathbb{R}$
can be expressed as $k(x) ~= ~cx^{\gamma} ~+~ \iota(x)$ with $\gamma > p$, $c \in \mathbb R$ and $\lim_{x \rightarrow 0} \frac{\iota(x)}{x^{\gamma}} = 0$. 

Then if $\alpha_{p}>0$, the series
\[
\sum_{n=0}^{\infty}{\frac{1}{p} \choose n}\left(\frac{k(x)}{\alpha_{p}x^{p}}\right)^{n}
\]
 is uniformly convergent for $x \in [0,r)$ for some $r>0$. 
\end{lem}
\begin{proof} The assumptions on $k$ imply that $\lim_{x\rightarrow0}\frac{k(x)}{x^{p}}=0$. Thus, there exists $r_{0}>0$ such that 
\begin{equation}
\label{inegalite1}
\left|\frac{k(x)}{\alpha_{p}x^{p}}\right| =\frac{k(x)}{\alpha_{p}x^{p}}<1 \mbox{ on } x \in [0,r_{0}).
\end{equation}

In addition, since $x \mapsto \frac{k(x)}{x^p}$ is dominated by $cx^\gamma$ in a neighborhood of $0$, it follows that in a neighborhood of $0$, the function $x \mapsto \frac{k(x)}{x^p}$ behaves the same as $cx^\gamma$ and thus is increasing around $0$. It follows that there exists $r_{1}>0$ such that for all $0 < x<y <r_{1}$,
\begin{equation} 
\label{inegalite2}
0 < \frac{k(x)}{x^p}  < \frac{k(y)}{y^p}.
\end{equation}
We define $r = \min(r_{0},r_{1})$. 

Since
$$
\left|{\frac{1}{p} \choose n}\right| \left|\frac{k(x)}{\alpha_{p}x^{p}}\right|^n \leq \left|\frac{k(x)}{\alpha_{p}x^{p}}\right|^n
$$ the series is convergent for $x \in [0,r)$. This inequality comes from the fact that the
sequence $\left|{\frac{1}{p} \choose n}\right|$ is decreasing in
$n$, which we prove at the end.  

We now show that the series is uniformly convergent in that interval.
Pick $0<\rho<r$. Then for $ 0< x<\rho<r,$ we have by \eqref{inegalite2} and \eqref{inegalite1}
\[
0 < \frac{k(x)}{x^{p}} <\frac{k(\rho)}{\rho^{p}}<\alpha_{p}.
\]
 Thus we have 
\[
\left|{\frac{1}{p} \choose n}\left(\frac{k(x)}{\alpha_{p}x^{p}}\right)^{n}\right|<\left|{\frac{1}{p} \choose n}\frac{k(\rho)}{\alpha_{p}\rho^{p}}^{n}\right|:=\left|{\frac{1}{p} \choose n}\right|\zeta^{n}\leq\zeta^{n},
\]
 with $0 < \zeta<1$. Consequently, the series $\sum_{n=0}^{\infty}|{\frac{1}{p} \choose n}|\zeta^{n}$
converges. Thus by the Weierstrass M-test, it follows that the series
converges uniformly for $0 \leq x<\rho$. Since $\rho$ is arbitrary in
$[0,r)$, the series converges uniformly on $[0,r)$. 

We now show that the  
sequence $\left|{\frac{1}{p} \choose n}\right|$ is decreasing in
$n$.  Let $u_{n}^{p}:=\left|{\frac{1}{p} \choose n}\right|$.
For $p=1$, we have $u_{0}^{1}=u_{1}^{1}=1$ and $u_{n}^{1}=0$ for
$n\geq2$. Thus, $\forall n\geq0,u_{n}^{1}\leq1$. Now, consider $p\geq2.$
By definition of the sequence, we have
\[
\forall k\in\mathbb{N},{\frac{1}{p} \choose 2k}=-u_{2k}^{p}\leq0,u_{2k-1}^{p}\geq0.
\]

For $k\geq1,$ we have
\begin{align*}
u_{2k+1}^{p}-u_{2k}^{p} & =\frac{\frac{1}{p}(\frac{1}{p}-1)(\frac{1}{p}-2)\cdots(\frac{1}{p}-2k)}{(2k+1)!}-\frac{-\frac{1}{p}(\frac{1}{p}-1)(\frac{1}{p}-2)\cdots(\frac{1}{p}-2k+1)}{(2k)!}\\
 & =\frac{\frac{1}{p}(\frac{1}{p}-1)(\frac{1}{p}-2)\cdots(\frac{1}{p}-2k+1)}{(2k)!}\left(1+\frac{\frac{1}{p}-2k}{2k+1}\right),\\
 & =\frac{\frac{1}{p}(\frac{1}{p}-1)(\frac{1}{p}-2)\cdots(\frac{1}{p}-2k+1)}{(2k)!}\left(\frac{1+\frac{1}{p}}{2k+1}\right)<0,
\end{align*}
since the numerator of the first fraction is the product of an even
number of terms, and thus is negative. 

Similarly, for $k\geq1,$ we have
\[
u_{2k}^{p}-u_{2k-1}^{p}=\frac{\frac{1}{p}(\frac{1}{p}-1)(\frac{1}{p}-2)\cdots(\frac{1}{p}-2k+2)}{(2k-1)!}\left(-\frac{\frac{1}{p}+1}{2k}\right)<0,
\]

since the numerator in the first fraction is the product of an odd
number of terms, which is positive. Thus

\[
\forall n\geq1,p\geq1,u_{n+1}^{p}-u_{n}^{p}\leq0.
\]
Since $u_{0}^{p}=1,\forall p\geq1$ and $u_{1}^{p}=\frac{1}{p},\forall p\geq1,$
it follows that 
\[
\forall n\geq0,p\geq1,u_{n+1}^{p}-u_{n}^{p}\leq0.
\]
Thus the sequence is decreasing and bounded above by its first term,
which is $1$. 

\end{proof}

\begin{lem}
\label{lem:xofeta}Assume that $g\in C^{2}([0,\infty),[0,\infty))$ with
$g(0)=0$, and for $p\in\mathbb{N}$,
$g^{(\nu)}(0)=0$ for $0\le\nu<p$ and \textcolor{black}{$g^{(p)}(0)>0$}. Let $F:[0,\infty)\mapsto[0,\infty)$
be defined as $F(x)=g(x)\sqrt{1+g'(x)^{2}}.$ Then locally around
$x=0$, $F$ is invertible and $F(x)=ax^{p}+o(x^{p})$ as $x\rightarrow0$
($a\in\mathbb{R}\setminus\left\{ 0\right\} )$ with $p\in\mathbb{N}$, and
\[
x(\eta)=F^{-1}(\eta)=\left(\frac{\eta}{a}\right)^{\frac{1}{p}}+o(\eta^{\frac{1}{p}})
\]
as $\eta\rightarrow0$. 
\end{lem}
\begin{proof}
Since $g\in C^{2}([0,\infty))$, then $F(x)=g(x)\sqrt{1+g'(x)^{2}}$
is in $C^{1}([0,\infty))$, with 
\[
F'(x)=\frac{g'(x)(1+g'(x)^{2})+g(x)g'(x)g''(x)}{\sqrt{1+g'(x)^{2}}}=\frac{g'(x)}{\sqrt{1+g'(x)^{2}}}(1+g'(x)^{2}+g(x)g''(x)).
\]

\begin{itemize}
\item Case 1. Corner case: suppose $g'(0)>0$. Then $F'(0)=\frac{g'(0)(1+g'(0)^{2})}{\sqrt{1+g'(0)^{2}}}>0$.
Since $F'$ is continuous on $[0,\infty),$ one can find a neighborhood
of $x=0$, $[0,\nu)$ (for some $\nu>0),$ such that for all $x\in[0,\nu),$
we have $F'(x)>0$. Thus, $F$ is strictly increasing on $[0,\nu)$
and therefore invertible on $[0,\nu)$, i.e.
\[
\forall x\in[0,\nu), x=F^{-1}(\eta)
\]
 with $F^{-1}(0)=0$ (since $g(0)=0\Rightarrow F(0)=0$). Since $F\in C^{1}([0,\infty))$
with $F'(0)>0$, it follows that $F$ has can be written as 
\[
F(x)=a_{1}x+q(x),
\]
 where $a_{1}=F'(0)>0$, and $\lim_{x\rightarrow0}\frac{q(x)}{x}=0.$
Thus $F(x)=a_{1}x+o(x)$ as $x\rightarrow0$ with $a_{1}\neq0$. Let
us look for its local inverse around $x=0$, $F^{-1}$, as
\begin{equation}
F^{-1}(\eta)=b_{1}\eta+h(\eta),\label{eq:inverseF_C2}
\end{equation}
 such that $\lim_{\eta\rightarrow0}\frac{h(\eta)}{\eta}=0$. Then
using $F^{-1}(F(x))=x$ for $x\in[0,\nu)$, we obtain
\begin{align*}
F^{-1}(F(x)) & =F^{-1}(a_{1}x+q(x))\\
 & =b_{1}(a_{1}x+q(x))+h(a_{1}x+q(x))\\
 & =a_{1}b_{1}x+b_{1}q(x)+h(a_{1}x+q(x))\\
 & =a_{1}b_{1}x+R(x),
\end{align*}
 where $R(x)= b_{1}q(x)+h(a_{1}x+q(x))$.
If we choose $b_{1}=\frac{1}{a_{1}}$, then for $x\in[0,\nu)$ we
have 
\[
F^{-1}(F(x))=x+R(x).
\]
 Now $F^{-1}$ is the correct inverse if $R(x)=0$ but we cannot know
what $R$ is since we do not have an expression for $q$ and $h$.
Nevertheless, asymptotically around $x=0$, $R$ needs to satisfy
$\lim_{x\rightarrow0}\frac{R(x)}{x}=0$. Let us then calculate this
limit:
\begin{align*}
\lim_{x\rightarrow0}\frac{R(x)}{x} & =\lim_{x\rightarrow0}\frac{b_{1}q(x)+h(a_{1}x+q(x))}{x}\\
 & = b_{1}\lim_{x\rightarrow0}\frac{q(x)}{x}+\lim_{x\rightarrow0}\frac{h(a_{1}x+q(x))}{x}\\
 & =0,
\end{align*}
 since $\lim_{x\rightarrow0}\frac{q(x)}{x}=0$
and $\lim_{x\rightarrow0}\frac{h(a_{1}x+q(x))}{x}=\lim_{y\rightarrow0}\frac{h(y)}{y}=0$
with $y=a_{1}x+q(x)=a_{1}x+o(x)\rightarrow_{x\rightarrow0}0$.
Thus if the inverse $F^{-1}$ is of the form (\ref{eq:inverseF_C2}),
we have the correct asymptotic behavior for $F^{-1}\circ F$. By the
unicity of the inverse, it follows that necessarily around $\eta=0$
we have 
\[
F^{-1}(\eta)=\frac{\eta}{a_{1}}+h(\eta),
\]
with $\lim_{\eta\rightarrow0}\frac{h(\eta)}{\eta}=0$.
\item Case 2: Cusp case: suppose $g^{(\nu)}(0)=0$ for $0\le\nu<p$ ($p \in \mathbb N, p \geq 2$) \textcolor{black}{and $g^{(p)}(0)>0$}. Then $F'(0)=0$. First, let's
show that $F$ is invertible and to do so, we will show that $F'(x)>0$
on an interval $(0,\mu),$ for some $\mu>0$. \textcolor{black}{Since $g \in C^2([0,\infty),[0,\infty))$ and satisfies $g(0)=g'(0) = 0$, it follows that in a neighborhood of $0$, say $(0,\mu)$ for some $\mu >0$, $g$ is strictly increasing and concave up. Thus $F'(x)>0$ on $(0,\mu)$ for some $\mu>0$. We therefore conclude that $F$ is invertible in a neighborhood of $0$.}

Now since $g\in C^{2}([0,\infty), [0,\infty))$ with $g(0)=g'(0)=0$, it follows
that $F \in C^{1}([0,\infty), [0,\infty))$ with $F(0)=F'(0)=0$. Let us assume
that we can write $F$ as follows
\[
F(x)=\alpha_{p}x^{p}+k(x),
\]
 with $\alpha_{p}>0$, and $k(x) = cx^{\gamma} + \iota(x)$ with $\gamma > p$, $c \in \mathbb R$ and $\lim_{x \rightarrow 0} \frac{\iota(x)}{x^{\gamma}} = 0$. Note that these assumptions imply that $\lim_{x\rightarrow0}\frac{k(x)}{x^{p}}=0$.
 
We then look for its local inverse $F^{-1}$
as
\[
F^{-1}(\eta)=\beta_{1}\eta^{\frac{1}{p}}+w(\eta),
\]
 such that $\lim_{\eta\rightarrow0}\frac{w(\eta)}{\eta^{\frac{1}{p}}}=0$.
Using $F^{-1}(F(x))=x$ on $(0,\mu)$, we obtain
\begin{align*}
F^{-1}(F(x)) & =\beta_{1}F(x)^{\frac{1}{p}}+w(F(x))\\
 & =\beta_{1}(\alpha_{p}x^{p}+k(x))^{\frac{1}{p}}+w(\alpha_{p}x^{p}+k(x))\\
 & =\beta_{1}\alpha_{p}^{\frac{1}{p}}x\sum_{n=0}^{\infty}{\frac{1}{p} \choose n}\left(\frac{k(x)}{\alpha_{p}x^{p}}\right)^{n}+w(\alpha_{p}x^{p}+k(x))\\
 & =\beta_{1}\alpha_{p}^{\frac{1}{p}}x+T(x),
\end{align*}
 where $T(x)=\beta_{1}\alpha_{p}^{\frac{1}{p}}x\sum_{n=1}^{\infty}{\frac{1}{p} \choose n} \left (\frac{k(x)}{\alpha_{p}x^{p}} \right )^{n}+w(\alpha_{p}x^{p}+k(x))$.
If we choose $\beta_{1}=\alpha_{p}^{\frac{-1}{p}}$, then we have
$F^{-1}(F(x))=x+T(x).$ It remains to show that this gives the correct
asymptotic behavior at $x=0$, namely that $\lim_{x\rightarrow0}\frac{T(x)}{x}=0$.
\begin{align*}
\lim_{x\rightarrow0}\frac{T(x)}{x} & =\lim_{x\rightarrow0}\frac{\beta_{1}\alpha_{p}^{\frac{1}{p}}x\sum_{n=1}^{\infty}{\frac{1}{p} \choose n} \left (\frac{k(x)}{\alpha_{p}x^{p}} \right )^{n}+w(\alpha_{p}x^{p}+k(x))}{x}\\
 & =\beta_{1}\alpha_{p}^{\frac{1}{p}}\lim_{x\rightarrow0}\sum_{n=1}^{\infty}{\frac{1}{p} \choose n} \left (\frac{k(x)}{\alpha_{p}x^{p}} \right )^{n}+\lim_{x\rightarrow0}\frac{w(\alpha_{p}x^{p}+k(x))}{x}\\
 & =\beta_{1}\alpha_{p}^{\frac{1}{p}}\sum_{n=1}^{\infty}{\frac{1}{p} \choose n}\frac{1}{\alpha_{p}^{n}}\lim_{x\rightarrow0} \left (\frac{k(x)}{x^{p}} \right )^{n}+\lim_{x\rightarrow0}\frac{w(\alpha_{p}x^{p}+k(x))}{x},
\end{align*}
 where we have used the uniform convergence of the series $\sum_{n=1}^{\infty}{\frac{1}{p} \choose n} \left (\frac{k(x)}{\alpha_{p}x^{p}} \right )^{n}$
to interchange the sum and the limit by Lemma \ref{lem:UnifConv}.
Since $\lim_{x\rightarrow0}\frac{k(x)}{x^{p}}=0$ it follows that for all $n\geq1,$ we have $\lim_{x\rightarrow0} \left (\frac{k(x)}{x^{p}} \right )^{n}=0$.
Additionally, $\lim_{x\rightarrow0}\frac{w(\alpha_{p}x^{p}+k(x))}{x}~=~\lim_{y\rightarrow0}\frac{w(y)}{y^{\frac{1}{p}}}~=~0$
with $y=\alpha_{p}x^{p}+k(x)=O(x^{p})\Rightarrow x=O(y^{\frac{1}{p}})$.
Therefore $\lim_{x\rightarrow0}\frac{T(x)}{x}=0$, hence leading to
the desired asymptotic behavior. Since the inverse is unique, it follows
that asymptotically around $\eta=0$ we have
\[
F^{-1}(\eta)=\left(\frac{\eta}{\alpha_{p}}\right)^{\frac{1}{p}}+w(\eta),
\]
 where $\lim_{\eta\rightarrow0}\frac{w(\eta)}{\eta^{\frac{1}{p}}}=0$.
\end{itemize}
 \end{proof}

\begin{proof}
(of Theorem \ref{thm:MainResult}). As shown earlier we have for $\eta>0$,
\[
l_{\eta}^{\pm}=A^{\pm}(x(\eta))+\eta B^{\pm}(x(\eta)).
\]
Since $f$ is continuous and $\kappa$ is continuous everywhere except
at the corner point, we have $ $$f(s,\pm g(s))\sqrt{1+g'(s)^{2}}=O(1)$
and $f(s,\pm g(s))\kappa(s)\sqrt{1+g'(s)^{2}}=O(1)$ as $x\rightarrow0$.
Thus
\begin{align*}
A^{\pm}(x(\eta)) & =\int_{0}^{x(\eta)}f(s,\pm g(s))\sqrt{1+g'(s)^{2}}ds\\
 & =O(\eta^{\frac{1}{p}})\mbox{ as }\eta\rightarrow0^{+},
\end{align*}
 and 
\begin{align*}
B^{\pm}(x(\eta)) & =\int_{0}^{x(\eta)}f(s,\pm g(s))\kappa(s)\sqrt{1+g'(s)^{2}}ds\\
 & =O(\eta^{\frac{1}{p}})\mbox{ as }\eta\rightarrow0^{+}
\end{align*}
by using Lemma \ref{lem:xofeta}. Thus
\[
l_{\eta}^{\pm}=O(\eta^{\frac{1}{p}})\mbox{ as }\eta\rightarrow0^{+}.
\]
 Therefore, using (\ref{eq:eq_corner}) , we have 
\[
I[\tilde{f},d](\eta)=I_{0}+O(\eta^{\frac{1}{p}})\mbox{ as }\eta\rightarrow0^{+}.
\]
For $\eta<0$, since everything is smooth on that side, we can apply
Theorem \ref{thm:Theorem1} to obtain
\[
I[\tilde{f},d](\eta)=I_{0}+O(\eta)\mbox{ as }\eta\rightarrow0^{-}.
\]

For $p=1$, we have
\begin{align*}
S & =\int_{\epsilon}^{\epsilon}I[\tilde{f},d](\eta)\delta_{\epsilon}(\eta)d\eta\\
 & =\int_{-\epsilon}^{\epsilon}I_{0}\delta_{\epsilon}(\eta)d\eta+O\left(\int_{-\epsilon}^{\epsilon}\eta\delta_{\epsilon}(\eta)d\eta\right)\\
 & =I_{0}+O\left(\int_{-\epsilon}^{\epsilon}\eta\delta_{\epsilon}(\eta)d\eta\right).
\end{align*}

In this case, $I[\tilde{f},d](\eta)$ only contains integer powers
of $\eta$ (Taylor series). It follows that since $\delta_{\epsilon}$
has $m$ vanishing moments, 
\begin{align*}
S & =I_{0}+O\left(\int_{\epsilon}^{\epsilon}\eta^{m+1}\delta_{\epsilon}(\eta)d\eta\right)\\
 & =I_{0}+O\left(\epsilon^{-k}\int_{-\epsilon}^{\epsilon}\eta^{m+1}d\eta\right)\\
 & =I_{0}+O(\epsilon^{m+2-k}).
\end{align*}

For $p\geq2,$ we have 
\begin{align*}
S & =\int_{-\epsilon}^{\epsilon}I[\tilde{f},d](\eta)\delta_{\epsilon}(\eta)d\eta\\
 & =\int_{-\epsilon}^{0}I[\tilde{f},d](\eta)\delta_{\epsilon}(\eta)d\eta+\int_{0}^{\epsilon}I[\tilde{f},d](\eta)\delta_{\epsilon}(\eta)d\eta\\
 & =\int_{-\epsilon}^{\epsilon}I_{0}\delta_{\epsilon}(\eta)d\eta+O\left(\int_{-\epsilon}^{0}\eta\delta_{\epsilon}(\eta)d\eta\right)+O\left(\int_{0}^{\epsilon}\eta^{\frac{1}{p}}\delta_{\epsilon}(\eta)d\eta\right)\\
 & =I_{0}+O\left(\epsilon^{-k}\int_{-\epsilon}^{0}\eta d\eta\right)+O\left(\epsilon^{-k}\int_{0}^{\epsilon}\eta^{\frac{1}{p}}d\eta\right)\\
 & =I_{0}+O(\epsilon^{-k+2})+O(\epsilon^{-k+\frac{1}{p}+1})\\
 & =I_{0}+O(\epsilon^{1+\frac{1}{p}-k}).
\end{align*}
\end{proof}

\section{Numerical simulations}

In this section, we present a few numerical computations aiming at
demonstrating the unique properties of the proposed approach to surface
integrals using implicit representations. These properties include:
\begin{enumerate}
\item High order approximations of smooth or piecewise smooth interfaces
with the use of sufficiently regular level set functions.
\item Analytically exact integrals and highly accurate numerical approximations
with the help of a sufficiently regular level set function and
constant-in-normal extensions of the integrands.
\item The potential in computing singular integrals using uniform Cartesian
grids.
\end{enumerate}
These properties are the consequences of the use of special averaging
kernels. The numerical computations presented in this section will
involve the following kernels, constructed in \cite{chen_tsai}:
\begin{itemize}
\item A $C^{\infty}$ kernel with one vanishing moment:
\[
\delta_{\infty,1}(r):=\exp\left(\frac{2}{(2r-1)^{2}-1}\right)(ar+b)\chi_{[0,1]}(r),
\]
\[
a=-261.5195892865372,b=145.7876577089403.
\]
\item A $C^{\infty}$ kernel with one vanishing moment:
\[
\delta_{\rho,\infty,1}(r):=\exp\left(\frac{1}{2(r^{2}-1)}\right)(a_\rho r+b_\rho)\chi_{[\rho,1]}(r).
\]
This kernel is designed specifically for integrands with an integrable singularity. The support of the kernel, which is $\rho$ away from the singularity, is constructed to mitigate the effect of the singularity. In the following computations, $\rho$ is taken to
be $0.1$, and  \[
a_{0.1}=-759.2781934172483,b_{0.1}=446.2604260472818.
\]
 
\item A $C^{\infty}$ kernel with two vanishing moments :
\[
\delta_{\infty,2}(r):=\exp\left(\frac{2}{(2r-1)^{2}-1}\right)(ar^{2}+br+c)\chi_{[0,1]}(r),
\]
$a=3196.1015220946833,b=-3457.6211113812255,c=852.9832518883903.$
\end{itemize}

\textcolor{black}{In the examples below, we use the Cartesian grids $h\mathbb{Z}^n\cap [-1,1]^n$, $n=2,3.$ The notation $\phi_{i,j}$ is used to denote the value of $\phi$ at the grid node $(ih,jh)$. Similar notations are used for other functions. }
\textcolor{black}{
If $\phi$ is a distance function, we define $\nabla\phi(x,y)\equiv 1.$ This is in fact an advantage of using distance functions to embed closed hypersurfaces.  We will study the effect of the kernels with more general Lipschitz continuous level set functions. 
Since our focus is on the kernels, we will use analytically defined formulas for $\nabla\phi$ in the respective examples.
}

\textcolor{black}{
Let us briefly summarize what to observe in the following examples.
The error computed by the proposed summation consists of two parts: one part is of analytical nature and depends on the number of vanishing moments of the kernel, as well as how
\[
I[f,\phi](\eta)=\int_{\Gamma_\eta} f dS,
\]
the one parameter family of integrals, changes as a function of $\eta$. Here $\Gamma_{\eta}$ is the $\eta$-level set of $\phi$.
If $I[f,\phi](\eta)$ is a quadratic polynomial in $\eta$, then there is no error in the analytical approximation for a kernel with at least two vanishing moments. Otherwise, the analytical error will be proportional to $\epsilon^{m+1}$, where $m$ is the number of vanishing moments of the kernel.
The second part of the 
computed error is the accuracy of the discretization of  the level set surface integral $S$, defined in \eqref{eq:S}. 
In the following examples, we propose the use of simple Riemann sum of the integrand over the grid. Due to the compactness of the kernel, the integrand in \eqref{eq:S}  can be regarded as a periodic function defined on $[-1,1]^n$, and thus the simple Riemann sum
is equivalent to the Trapezoidal rule. Therefore, the discretization errors inherits the convergence property of the Trapezoidal rule 
for periodic functions on uniform grids. The regularity of the kernel greatly influence the discretization error.}
\textcolor{black}{
Finally, we pointed out that in practice, if $\phi$ is not a distance function,  $\nabla\phi$ will need to be approximated on the grid. 
Typically, $\nabla\phi$ is approximated by simple central differencing yielding second order in $h$ approximations or fifth order WENO approximations. Naturally the accuracy of the overall approximation of the surface integral will be affected.
\emph{In other words, we show that by a smart choice of kernel, we can eliminate many components of the errors associated to a typical approximation of  \eqref{eq:surface_integral} by discretization of \eqref{eq:S} on uniform grids.}
}

\textcolor{black}{
Also in the examples below, we study the relative errors as functions of $h$,  
and present the observed rate at which the computed errors tend to $0$ as $h$ decreases.
We shall scale $\epsilon$ according to different powers of $h$ and observe that
the computed error does decrease at the rate as predicted by the theory, when discretization errors are negligible (due to $C^\infty_c$ kernels).
We point out that when $\epsilon$ is too small compared to $h$, the analytical theory that we develop in this paper no longer holds, as the problem becomes purely discrete. Such a scenario is analyzed in \cite{engquist_tornberg_tsai05}.
}

\begin{example}
In this example, we compute \textcolor{black} {integrals on the circle} $x^{2}+y^{2}=r_{0}^{2}$, for $r_0=0.501$.
\textcolor{black}{We first study the accuracy of the extrapolative approach in computing the length
of the circle, without using the distance function. The approximations are computed by the formula 
 \[
S_{N}:=\sum_{i,j}\epsilon^{-1}\delta_\epsilon(\epsilon^{-1}\phi_{i,j})|\nabla\phi_{_{i,j}}|\,h^{2},
\]
where  $\phi(x,y)=x^{2}+y^{2}-r_{0}^{2}$ and $\nabla\phi_{i,j}$ is evaluated with the exact formula of
$\nabla\phi(ih,jh)$.}
 In Table~\ref{tab:circle-nondistance-fn}, we present our computations
with radius $r_{0}=0.501$, and $\epsilon=2h^{1/2}$ with $h~=~2/N,$
where $N$ is the number of grid points along one coordinate direction.
\textcolor{black}{The analytical approximation error  is $\mathcal{O}(\epsilon^2)$ for $\delta_\epsilon=\delta_{\infty,1}$, and $\mathcal{O}(\epsilon^3)$  for 
$\delta_{\infty,2}$, due to the number of vanishing moments that each kernel has.
Provided that the corresponding quadrature error is negligible, with the scaling $\epsilon=2h^{1/2}$, one should observe that, as $h\rightarrow 0$, the convergence rates for using $\delta_{\infty,1}$
is $1$, and  $1.5$ for using  $\delta_{\infty,2}$. }

We point out that with the same $\epsilon$ at $N=100,$ if we used
the signed distance function, $\phi(x,y)=\sqrt{x^{2}+y^{2}}-r_{0}$
to the same circle, the relative error using $\delta_\epsilon=\delta_{\infty,1}$
would be $2.31890e-08$. This reflects the property of the integral
$I[\tilde{f}\equiv1,\phi](\eta)$, defined in (\ref{eq:familyfunc}),
as a function of $\eta.$ 

\textcolor{black}{Next, we study the accuracy of the extrapolative approach in integrating a Lipschitz continuous function defined on
the same circle. In the computation, the level set function is the signed distance function $d(x,y) = \sqrt{x^2+y^2}-r_0$, and the integrand $f$ is defined by
\[
f(x,y)= \min(|\theta -0.3|, |\theta \textcolor{black}{-} 2\pi-0.3|),~~~0\le \theta=\arg(x,y)<2\pi. 
\]
$f$ is not differentiable along $\theta=0.3$ and \textcolor{black}{$\theta = \pi+0.3$.} 
In Figure~\ref{fig:lip-f} we present relative errors computed by the sum:
\[
S_{N}[f]:=\epsilon^{-1}\sum_{i,j} f(ih,jh) \delta(\epsilon^{-1}\phi_{i,j})\,h^{2},
\]
with $\epsilon=2\sqrt{h}=2/\sqrt{N}$, and $\delta_\epsilon=\delta_{\infty,2}.$
We observe the fast exponential convergence of the relative errors. On the one hand, we  acknowledge that
\[
I[f,\phi](\eta)=\int_{\Gamma_\eta} f dS,
\]
is  a degree one polynomial in $\eta$, even though $f$ is only Lipschitz continuous. The chosen kernel $\delta_{\infty,2}$ has enough resolving power, and the analytical error of the approximation is zero.
It is surprising, however, that the discretization errors converge so fast, even when the integrand is formally only Lipschitz continuous.
}

\begin{table}[h]
\begin{centering}
\caption{Relative error in computing the circumference of a circle using a
non-distance level set function. \label{tab:circle-nondistance-fn}}
\par\end{centering}
\medskip{}

\centering{}%
\begin{tabular}{|c|c|c|c|c|c|c|}
\hline 
 & N=100 & 200 & 400 & 800 & 1600 & 3200\tabularnewline
\hline 
\hline 
Rel. err. $\delta_{\infty,1}$ & 2.19034e-02 & 1.22417e-02 & 6.72509e-03 & 3.61084e-03 & 1.90462e-03 & 9.90744e-04\tabularnewline
\hline 
Order &  & 0.8 & 0.9 & 0.9 & 0.9 & 0.9\tabularnewline
\hline 
\hline 
Rel. err. $\delta_{\infty,2}$ & 2.99384e-03 & 1.53839e-03 & 6.34199e-04 & 2.55519e-04 & 9.96251e-05 & 3.78689e-05\tabularnewline
\hline 
Order &  & 1.0 & 1.3 & 1.3 & 1.4 & 1.4\tabularnewline
\hline 
\end{tabular}
\end{table}
\end{example}
\medskip{}

\begin{figure}[h]
\begin{centering}
\includegraphics[scale = 1.0]{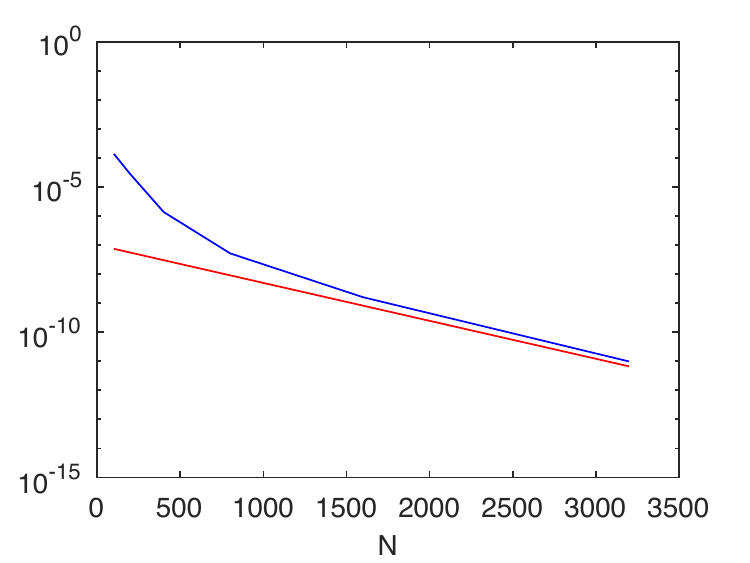}
\par\end{centering}
\caption{The blue curve reveals  the relative errors computed by $S_N[f]$ for integrating a Lipschitz function on a circle. The red curve shows the graph of $ 0.997^N10^{-7}$. \label{fig:lip-f}}
\end{figure}

\begin{example} \label{ex_cusp}
In this example, we compute the length of the black curve shown in
Figure~\ref{fig:Cusp}, which has four cusps. 
\begin{figure}[h]
\begin{centering}
\includegraphics[scale = 1.0]{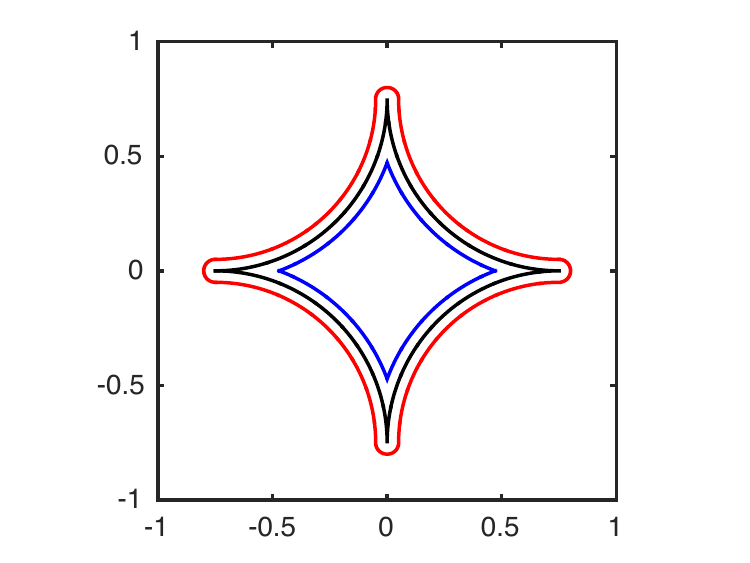}
\par\end{centering}
\caption{Cusp\label{fig:Cusp}}
\end{figure}
The curve $\Gamma_{0}$
is defined by the four quarter of circles with radius $r_{0}=0.75,$
centered respectively as $(r_{0},0),(-r_{0},0),(0,r_{0})$ and $(0,-r_{0}).$ 

The length is computed by the formula 
\[
S_{N}^{+}:=\sum_{i,j}\epsilon^{-1}\delta_{\infty,1}(\epsilon^{-1}d_{i,j})\,h^{2}
\]
using $\epsilon=0.05$. The relative errors are tabulated in Table~\ref{tab:cusp-errors}. The convergence in this example is actually exponential, namely the error decays like $\alpha^N$, with $0 < \alpha< 1$. In this example, $\alpha \approx 0.9954$. Figure~\ref{fig:cusp_convergence} shows the exponential convergence rate.
Here $d_{i,j}$ denotes the value of the signed distance function
to $\Gamma_{0}$ at the point $(x_{i},y_{j}).$ In the computations, the
sign of the signed distance function to $\Gamma_{0}$ is designated to
be negative inside the enclosed region. Note that the level sets of
$d$ in the region $\{d>0\}$ are continuously differentiable. As
the analysis in the above section shows, computations performed in
$\{d<0\}$ using
\[
S_{N}^{-}:=\sum_{i,j}\epsilon^{-1}\delta_{\infty,1}(-\epsilon^{-1}d_{i,j})\,h^{2},
\]
will not yield accurate approximations due to the singularity of $I[\tilde{f}\equiv1,\phi](\eta)$
near $\eta=0^{-}.$

\begin{table}[h]
\begin{centering}
\caption{Relative errors in computing the length of the black curve, containing
four cusps, in Figure \ref{fig:Cusp}. \label{tab:cusp-errors}}
\par\end{centering}
\medskip{}

\centering{}%
\begin{tabular}{|c|c|c|c|c|c|c|}
\hline 
$w_{\infty,1}$ & N=100 & 200 & 400 & 800 & 1600 & 3200\tabularnewline
\hline 
\hline 
Rel. error $S_{N}$ & 7.04018e-3 & 6.63514e-4 & 4.43853e-5 & 4.45564e-7 & 5.84085e-9 & 3.74043e-12\tabularnewline
\hline 
Order &  & 3.4 & 3.9 & 6.6 & 6.3 & 10.6\tabularnewline
\hline 
\end{tabular}
\end{table}

\begin{figure}[h]
\begin{centering}
\includegraphics[scale = 1]{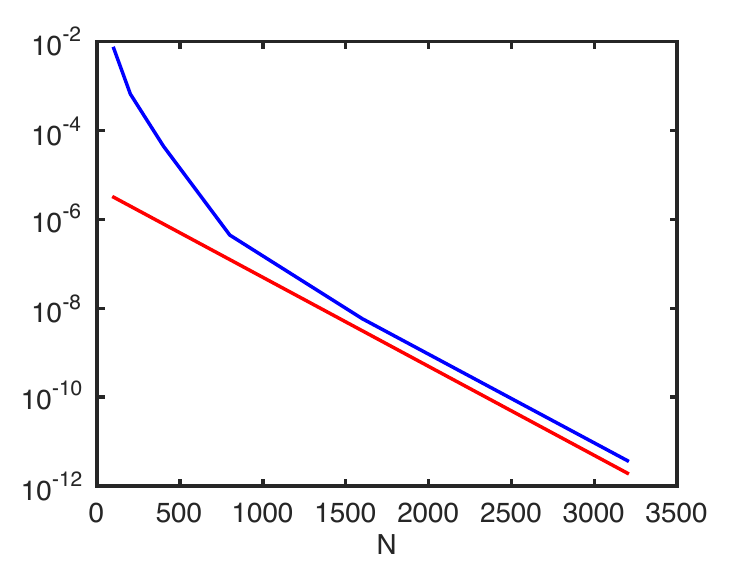}
\par \end{centering}
\caption{The blue curve reveals the relative errors computed by $S_N[f]$ for computing the length of the black curve with four cusps in Example~\ref{ex_cusp}. The red curve shows the graph of $ 5(0.9954)^N10^{-6}$. \label{fig:cusp_convergence}}
\end{figure}

In Table~\ref{tab:cusp-shifted-errors}, we present the numerical errors
computed by $\tilde{S}_{N}^{-}$ to approximate the length of the
interface which is $0.05$ distance away from the black curve shown
in Figure~\ref{fig:Cusp}
\[
\tilde{S}_{N}^{-}:=\sum_{i,j,k}\epsilon^{-1}\delta_{\infty,1}(-\epsilon^{-1}(d_{i,j}+0.05))\,h^{2}.
\]
 
\begin{table}[h]
\begin{centering}
\caption{Relative errors in computing the length of the blue curve, containing
four corners, in Figure \ref{fig:Cusp}. The theoretical convergence
rate for this simulation is $2.0.$ \label{tab:cusp-shifted-errors}}
\par\end{centering}
\medskip{}

\centering{}%
{\small \begin{tabular}{|c|c|c|c|c|c|c|c|}
\hline 
$w_{\infty,2}$ & $\epsilon$ & N=100 & 200 & 400 & 800 & 1600 & 3200\tabularnewline
\hline 
\hline 
Rel. error $S_{N}$ & $3.4N^{-2/3}$ & 1.64925e-02 & 8.63529e-03 & 2.98334e-03 & 1.08381e-03 & 3.34617e-04 & 9.79520e-05\tabularnewline
\hline 
Order &  &  & 0.9 & 1.5 & 1.5 & 1.7 & 1.8\tabularnewline
\hline 
\end{tabular}}
\end{table}
\medskip{}

\end{example}

\begin{example} \label{ex_ell1}
We compute the surface area of $\phi(x,y,z):=|x|+|y|+|z|=r_{0}$ (graphed in Figure~\ref{fig:ell1_ball})
with
$r_{0}=0.65$ by the following sum
\[
S_{N}:=\sum_{i,j,k}\epsilon^{-1}\delta_{\infty,2}(\epsilon^{-1}\phi_{i,j,k})|\nabla\phi_{i,j,k}|h^{3},
\]
where $|\nabla\phi_{i,j,k}|\equiv\sqrt{3}.$ 

\begin{figure}[h]
\begin{centering}
\includegraphics[scale = 1]{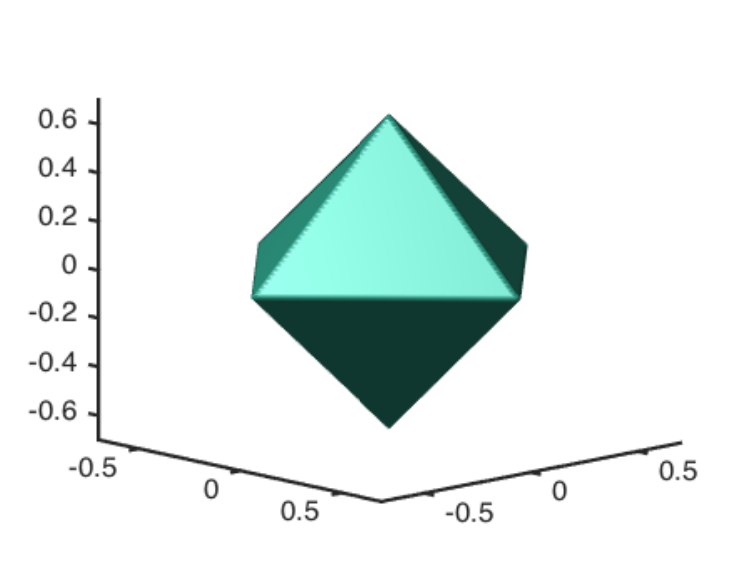}
\par \end{centering}
\caption{Surface area of $|x|+|y|+|z|=r_{0}$ computed in Example~\ref{ex_ell1}. \label{fig:ell1_ball}}
\end{figure}

The relative errors with
$\epsilon=0.1$ and a few values of $h=1/N$ are presented in Table~\ref{tab:L1-ball}. The convergence rate in this simulation is also exponential and is illustrated in Figure~\ref{fig:cusp_convergence}.
The point of this example is to demonstrate that the proposed approach
is able to compute high order approximations of surface integrals,
in the case where the surface and the embedding level set function
are only piecewise smooth. This capability is not seen in other existing
level set methods.

\begin{table}[h]
\begin{centering}
\caption{Relative error in computing the surface area of an $\ell_{1}$-ball.
\label{tab:L1-ball}}
\par\end{centering}
\medskip{}

\centering{}
\begin{tabular}{|c|c|c|c|c|}
\hline 
 & N=100 & 200 & 400 & 800\tabularnewline
\hline 
\hline 
Rel. error & 5.87232e-1 & 2.63126e-2 & 8.19894e-4 & 5.23091e-6\tabularnewline
\hline 
Order &  & 4.5 & 5.0 & 7.3\tabularnewline
\hline 
\end{tabular}
\end{table}

\medskip{}

\begin{figure}[h]
\begin{centering}
\includegraphics[scale = 1]{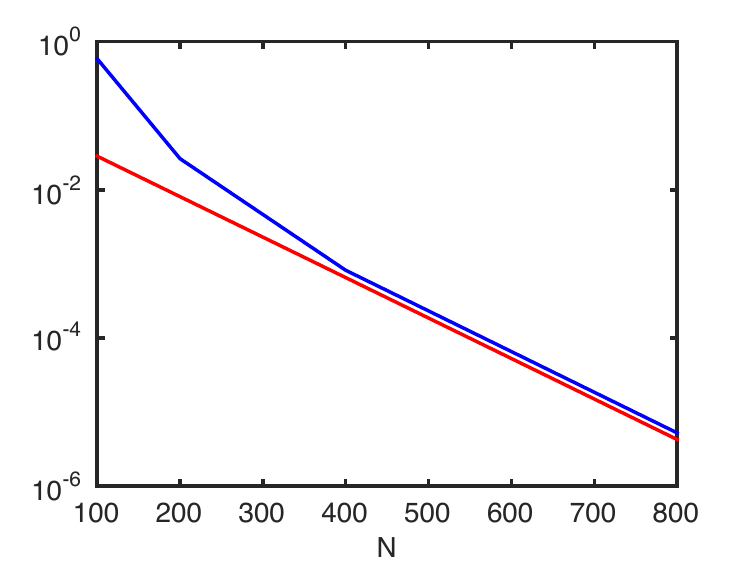}
\par \end{centering}
\caption{The blue curve reveals the relative errors computed by $S_N[f]$ for computing the surface area of the $\ell^1$ ball described in Example~\ref{ex_ell1}. The red curve shows the graph of $0.9875^N10^{-1}$. \label{fig:ell1_ball}}
\end{figure}

\end{example}

\begin{example}
In this example, we compute the line integral of a function that has
an integrable singularity at a corner of the interface. \textcolor{black}{Let
$\phi^{(1)}(x,y):=|x|+|y|-1,$ $\phi^{(2)}(x,y)=\text{sgn}(\phi^{(1)}(x,y))\left(\phi^{(1)}(x,y)\right)^{2}$,
where $\text{sgn}(z)$ is the signum function. Thus $\Gamma_{0}$ is the
$\ell_{1}$-ball. Analytically, the gradients of $\phi^{(1)}$ and $\phi^{(2)}$ exist almost everywhere. 
So in our computation, we globally define the gradients to be
$|\nabla\phi_{i,j}^{(1)}|=\sqrt{2}$
and $|\nabla\phi_{i,j}^{(2)}|=\sqrt{8|\phi_{i,j}^{(2)}|}.$ }
 Let $f(r)=1/\sqrt{r}$, for $r\neq 0$ and $f(0)=10^9$.  We define the interface
$\Gamma_{0}:=\{\phi=0\}$ and we approximate the integral
\[
\int_{\Gamma_{0}}f(\sqrt{x^{2}+(y-1)^{2}})dS(x,y)
\]
 by 
\[
S_{N}^{(\ell)}:=\sum_{i,j}f(\sqrt{x_{i}^{2}+(y_{j}-1)^{2}})\epsilon^{-1}\delta_{\epsilon}(\epsilon^{-1}\phi_{i,j}^{(\ell)})|\nabla\phi_{i,j}^{(\ell)}|h^{2},
\]
 where \textcolor{black}{$\delta_{\epsilon}=\delta_{0.1,\infty,1}$ . 
In Table~\ref{tab:singularity-corner-example}, we present numerical
errors of our computations, where $\epsilon$ is adjusted for each
kernel so that the same number of points are used in the narrow bands
$\{0.1\epsilon<|\phi_{i,j}^{(1)}|<\epsilon\}$ and $\{0.1\epsilon<|\phi_{i,j}^{(2)}|<\epsilon\}$. }

\textcolor{black}{This is an example that suggests the potential of
the proposed extrapolative approach in computing integrands involving
integrable singularities. For integration of singularities such as
$1/\sqrt{x}$ in the interval $[0,1]$, one typically needs to require
that the step size $h=h(x)$ decreases sufficiently fast as $x$
tends to $0,$ otherwise, the resulting quadrature will have a significant
drop in the order of accuracy. However, as Table~\ref{tab:singularity-corner-example}
shows, the relative errors computed with $\phi^{(1)}$ decrease very
slowly, but they are all under 1\%; the relative errors computed with
$\phi^{(2)}$ decrease steadily as the mesh refines. }

\textcolor{black}{There are three factors that determines the performance
of the algorithm. The first one is the use of an extrapolative
kernel which tends to zero as the value of the level set function
goes to zero. Therefore, the singularity of the integrand is suppressed. Second, we observe that near $(0,1)$,
the singularity of $f$, $|\nabla\phi^{(2)}(x,y)|$ is small, and
the product does not have a singularity. See Figure~\ref{fig:nonlinear-sampling}
for the graph of $f\,|\nabla\phi^{(2)}|$. However, the derivative
of $f\,|\nabla\phi^{(2)}|$ is unbounded at $(0,1)$, therefore, we expect that typical methods based on uniform grids will have deterioration in the accuracy and the convergence rate. However, the results reported in Table~\ref{tab: singular_int} are surprisingly better. Finally, the level set function
$\phi^{(2)}$ is proportional to the squared distance to the interface
$\Gamma$ (the zero level set of $\phi^{(2)}$), and when discretized
with a uniform grid, one has effectively an adaptive meshing with
quadratic refinement in the step size in the direction normal to the
interface. See Figure~\ref{fig:nonlinear-sampling} for an illustration
comparing the discretization resulting from $\phi^{(1)}(0,jh)$ and
$\phi^{(2)}(0,jh)$.} 

\textcolor{black}{Rigorous analysis of our approach to this type of
singular integrals is the subject of another paper.}

\begin{figure}
\begin{centering}
\includegraphics{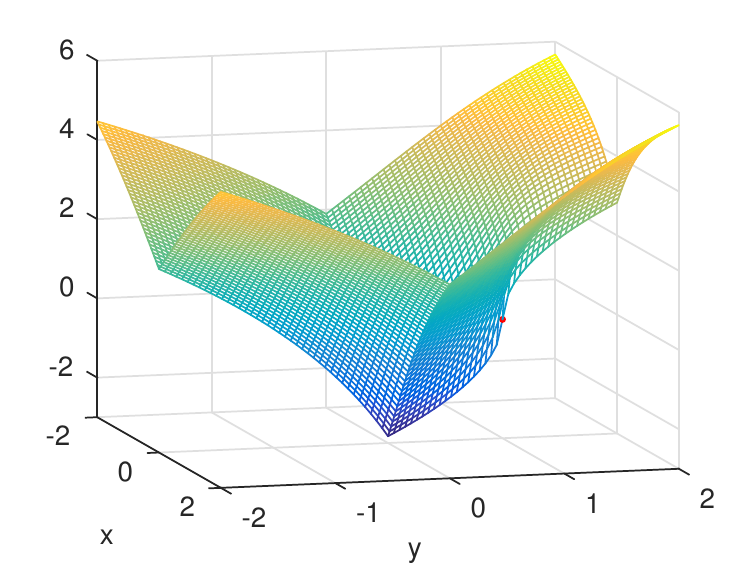}\includegraphics{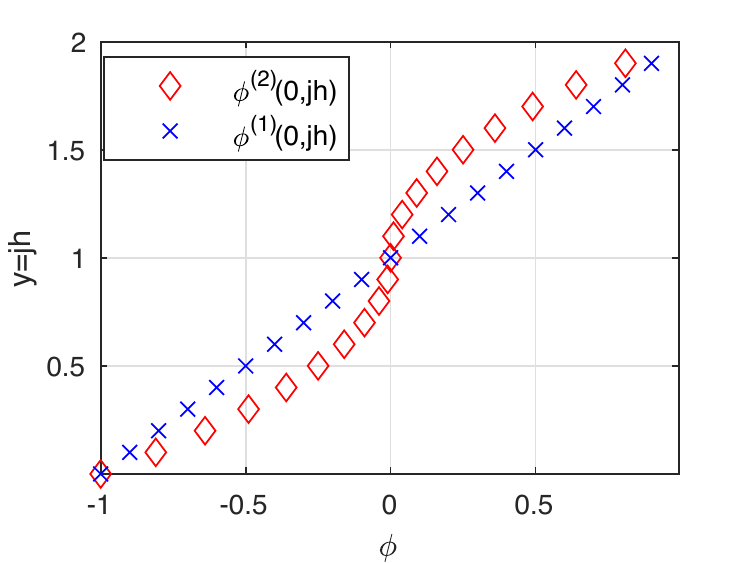}
\par\end{centering}
\caption{\label{fig:nonlinear-sampling} Left: The graph of $f(x,y)|\nabla \phi^{(2)}(x,y)| := f(x,y) \sqrt{8\phi^(2)(x,y)}$. The red dot indicates where the integrand is singular.  Right: $\{\phi^{(1)}(0,jh):j=0,\pm1,\pm2,\cdots\}$
and $\{\phi^{(2)}(0,jh):j=0,\pm1,\pm2,\cdots\}$ have different spacings.
\textcolor{black}{The density of the latter increases as the absolute value
of $|\phi^{(2)}(0,jh)|$ is closer to $0$. }}

\end{figure}

\begin{table}
\begin{centering}
\caption{Relative errors in the computation of a singular integral. \textcolor{black}{The
observed slow convergence resulting from using $\phi^{(1)}$ is expected,
as the grid cannot resolve the singularity in $f\,|\nabla\phi^{(1)}|$.
On the contrary, the quadrature with $\phi^{(2)}$ does better, at
least rate-wise, and the corresponding computation produces decreasing
errors as $N$ increases. }\label{tab:singularity-corner-example}}
\label{tab: singular_int}

\par\end{centering}
\medskip{}

\centering{}%
\begin{tabular}{|c|c|c|c|c|c|c|}
\hline 
Ker=$w_{\delta,\infty,1}$ & $\epsilon$ & N=200 & 400 & 800 & 1600 & 3200\tabularnewline
\hline 
\hline 
Rel. err. $\phi^{(1)}$ & $a_{0}N^{-0.475}$ & 1.01552e-02 & 8.84065e-03 & 7.63649e-03 & 6.55206e-03 & 5.59749e-03\tabularnewline
\hline 
Order &  &  & 0.2 & 0.2 & 0.2 & 0.2\tabularnewline
\hline 
\hline 
Rel. err. $\phi^{(2)}$ & $a_{0}^{2}N^{-0.95}$ & 1.77161e-02 & 9.47018e-03 & 4.62084e-03 & 1.51821e-03 & 4.30993e-04\tabularnewline
\hline 
Order &  &  & 0.9 & 1.0 & 1.6 & 1.8\tabularnewline
\hline 
\end{tabular}
\end{table}
\end{example}

\section{Discussion}

In this section, we compare this new approach with the original KTT
approach constructed in~\cite{kublik_tanushev_tsai13} and discuss
the potentials of the extrapolative approach for hypersurfaces with
singularities. For smooth hypersurfaces, the original and the extrapolative
approach both yield an exact result, namely the volume integral coincide
exactly with the hypersurface integral one wishes to compute. However,
exactness is not obtained the same way. The original approach needs
a Jacobian term in the integrand: this Jacobian corrects for the change
in curvature incurred when one moves away from the hypersurface (namely
the zero level set of the level set function). The extrapolative approach
does not have a Jacobian but instead requires the approximation of
the Dirac delta function to have at least $n-1$ vanishing moments,
where $n$ is the dimension. While the choice of the method is ultimately
up to the practitioner, we believe it is easier to use the original
approach on smooth hypersurfaces since (a) the Jacobian is easy to
compute using the singular values of the Jacobian matrix of the closest
point mapping (see \cite{kublik_tsai16}), and (b) there is no need
to construct kernels with large numbers of vanishing moments, other
than accuracy gain.

On the other hand, while the original approach suffers from low accuracy
when used on hypersurfaces with singularities, the extrapolative approach
is better suited for such cases. That said, we point out that for
hypersurfaces with singularities, neither one of these two approaches
will yield an integral formulation that coincides exactly with the
hypersurface integral. However, the extrapolative approach is able
to achieve good accuracy on hypersurfaces with corners, while the
original approach does not perform well on any hypersurfaces with
singularities. Unlike the original approach, the extrapolative one
looks at a family of functionals in $\eta$, where $\eta$ is the
distance from a shifted level set and the hypersurface. In that case,
if the singularity is a corner, the family of functionals will be
a polynomial in $\eta$ and thus, the accuracy will depend on the number
of vanishing moments of the kernel. This approach is therefore capable
of achieving high accuracy for computations of integrals over piecewise
smooth hypersurfaces. For cusps, the family of functionals is not
polynomial in $\eta$ but a series with fractional powers in $\eta.$ In that case, our analysis
suggests that it is necessary to construct a different
class of kernels with ``fractional vanishing moments'' in order to achieve
high accuracy. 
Nevertheless, this extrapolative approach provides a stepping
stone towards computing over hypersurfaces with singularities. In
addition, we have shown a numerical simulation that suggests the potential of this technique for integrating functions with integrable singularities. 
\section{Conclusion}

We described an extrapolative approach for integrating over hypersurfaces in
the level set framework. This method is based on the classical integral formulation
using an approximation of the Dirac delta function typically used with
level sets. This analytical integral formulation is for most cases
an approximation of the integral one wishes to compute. We show that
for smooth interfaces, if the kernel approximating the Dirac delta
function has enough vanishing moments, the integral formulation
is actually equal to the hypersurface integral. In addition, unlike
previous numerical integration schemes for level sets, we demonstrate
that this method is capable of computing a line or surface integral
with very high accuracy in the case where the hypersurface is only
piecewise smooth (e.g. with corners). Finally, with an appropriate choice
of kernel approximating the Dirac delta function, we can also compute
integrals where the integrand has an integrable singularity. In particular, this
work lays the foundation of a numerical scheme for computing general
improper integrals. 

\section*{Acknowledgement}

Kublik is supported by a University of Dayton Research Council Seed
Grant. Tsai is supported partially by a National Science Foundation
Grant DMS-1318975 and an ARO Grant No. W911NF-12-1-0519. Tsai also thanks National Center for Theoretical
Sciences, Taiwan, for hosting his stay at the center where part of
the research for this paper was conducted. 

\bibliographystyle{plain}
\bibliography{references}

\end{document}